\documentclass[11pt]{article}
\usepackage[margin=1in]{geometry}
\usepackage[english]{babel}

\usepackage{amsmath}
\usepackage{amsthm}
\usepackage{graphicx}
\usepackage{amsfonts}
\usepackage{amssymb}
\usepackage{enumerate}
\usepackage{comment}
\usepackage[dvipsnames]{xcolor}

\usepackage{upgreek}
\usepackage{todonotes}
\usepackage{stmaryrd}

\usepackage{url}
\usepackage{hyperref} 

\hypersetup{
	colorlinks   = true,
	citecolor    = red,
	linkcolor    = blue,
	urlcolor     = orange
}
\usepackage{cleveref}

\usepackage[square,numbers]{natbib}
\usepackage[normalem]{ulem}
\usepackage{bbm}

\DeclareOldFontCommand{\bf}{\normalfont\bfseries}{\mathbf}
\DeclareOldFontCommand{\cal}{\normalfont\bfseries}{\mathcal}

\newtheorem{theorem}{Theorem}[section]
\newtheorem{lemma}[theorem]{Lemma}
\newtheorem{proposition}[theorem]{Proposition}
\newtheorem{corollary}[theorem]{Corollary}

\newtheorem{assumption}{Assumption}[section]
\newtheorem{example}{Example}[section]
\newtheorem{definition}{Definition}[section]
\newtheorem{remark}{Remark}[section]

\numberwithin{equation}{section}

\newcommand{\Eb}{\mathbb{E}}
\newcommand{\Nb}{\mathbb{N}}
\newcommand{\Pb}{\mathbb{P}}
\newcommand{\Rb}{\mathbb{R}}
\newcommand{\Tb}{\mathbb{T}}

\newcommand{\Dc}{\mathcal{D}}
\newcommand{\Nc}{\mathcal{N}}
\newcommand{\Pc}{\mathcal{P}}

\newcommand{\Pf}{\mathfrak{P}}
\newcommand{\Ff}{\mathfrak{F}}

\newcommand{\bP}{\mathbf{P}}

\newcommand{\dd}{\mathrm{d}}

\title{Mean-Field Games Under Model Uncertainty}

\author{Zongxia Liang\thanks{Department of Mathematical Sciences, and Center for Insurance and Risk Management, School of Economics and Management, Tsinghua University, China. Email: liangzongxia@tsinghua.edu.cn} 
	\hspace{2ex}
Zhou Zhou\thanks{School of Mathematics and Statistics, University of Sydney, Australia. Email: zhou.zhou@sydney.edu.au} 
\hspace{2ex}
Yaqi Zhuang\thanks{Department of Mathematical Sciences, Tsinghua University, China. Email: zyq23@mails.tsinghua.edu.cn} 
\hspace{2ex}
Bin Zou\thanks{Department of Mathematics, University of Connecticut, USA. Email: bin.zou@uconn.edu}}

\date{\today}

\begin{document}
\maketitle

\begin{abstract}
We study discrete-time, finite-state mean-field games (MFGs) under model uncertainty, where agents face ambiguity about the state transition probabilities. Each agent maximizes its expected payoff against the worst-case transitions within an uncertainty set. Unlike in classical MFGs, model uncertainty renders the population distribution flow stochastic. This leads us to consider strategies that depend on both individual states and the realized distribution of the population. Our main results establish the asymptotic relationship between $N$-agent games and MFGs: every MFG equilibrium constitutes an $\varepsilon$-Nash equilibrium for sufficiently large populations, and conversely, limits of $N$-agent equilibria are MFG equilibria. We also prove the existence of equilibria for finite-agent games and construct a solvable mean-field example with closed-form solutions. 
\end{abstract}


\section{Introduction}
\label{sec:intro}

\subsection{Background and Motivation}
Mean-field game (MFG) theory was independently developed by \citet{LasryLions.06a, LasryLions.06b,lasry2007mean} and \citet{huang2006large, HuangMalhameCaines.07}, providing a mathematical framework for analyzing strategic interactions among large populations of agents. 
The central premise of the MFG theory is that agents in the population share similar characteristics and objectives, and each agent's optimal decision depends primarily on the aggregate distribution of all agents' states rather than on the specific actions of individual agents. This ``mean-field'' approximation transforms the original high-dimensional multi-agent game into a tractable optimal control problem, significantly reducing computational complexity while preserving the essential features of the original game. 
Over the past two decades, MFGs have attracted substantial research interest across mathematics, economics, and engineering, with applications ranging from financial markets and energy systems to crowd dynamics and epidemiology; see the monographs \citet{CarmonaDelaRue.17a, CarmonaDelaRue.17b, bensoussan2013mean} and the survey paper \citet{gomes2014mean} for a comprehensive treatment of MFGs.

A classic topic of discrete-time mean-field problems is the mean-field Markov game, which has been extensively studied in the existing literature. \citet{gomes2010discrete} pioneered the study of discrete-time MFGs on finite state spaces, establishing existence and uniqueness results. \citet{saldi2018markov} introduced the solution concept of Markov-Nash equilibrium under infinite-horizon discounted cost criteria; they proved the existence of mean-field equilibria under mild assumptions and demonstrated that these equilibria serve as approximate solutions for finite-agent games. Building on these foundational results, subsequent research has expanded the theoretical framework to partially observable environments and ergodic cost criteria; see, e.g.,  \citet{saldi2019approximate} and  \citet{biswas2015mean}. However, these studies assume that agents have complete knowledge of state transition probabilities. This assumption is often unrealistic, as the mechanisms governing state transitions in real economic and social systems typically involve significant uncertainty.

To address model uncertainty in decision-making, scholars have developed theoretical frameworks across economics and operations research that inform our approach. In economics, \citet{hansen2001robust, hansen2008} introduced robust control theory to address uncertainty in macroeconomic model parameters, while \citet{gilboa1989maxmin} proposed maxmin expected utility theory to model decision-making under ambiguity with multiple priors.
In operations research and control theory, \citet{wiesemann2013robust} systematically studied robust Markov decision processes, and \citet{iyengar2005robust} analyzed dynamic programming under transition probability uncertainty. Specifically, for transition uncertainty in Markov decision processes, \citet{nilim2005robust} and \citet{xu2010distributionally} proposed methods based on uncertainty sets, where the key idea is to optimize against the worst-case scenario within a set of plausible probability distributions.

In this paper, we introduce robust optimization into the MFG framework to study discrete-time MFGs with model uncertainty on finite state spaces. Each agent adopts a robust decision strategy to maximize the expected payoff under the worst-case transition probabilities. However, such a robust extension introduces fundamental challenges, as we explain in Subsection~\ref{sub:framework}.

\subsection{Standard Framework and Notation}
\label{sub:framework}

We study several closely related versions of a Markov game, and all of them share the same standard framework. In this subsection, we introduce the basic setup, mainly based on the infinite-horizon MFG version, as well as the standard notation that will be used in this paper.

Let $\Nb = \{0, 1, \cdots\}$ denote the set of non-negative integers and $\Nb_+ := \Nb \backslash \{0\}$. 
Let $\Dc := \{1,2,\cdots,d\}$ (for some $d \in \Nb_+$ denote the finite state space with $d$ states for the Markov state process of every agent, equipped with the discrete metric $d_\Dc$. Let $U$ denote the action space, which is a compact set in $\Rb^{\ell}$ (for some $\ell \in \Nb_+$).  We use $\Tb$ to denote the discrete time horizon, which may be infinite, i.e., $\Tb :=\{0,1,2,\cdots\}$, or finite, i.e., $\Tb :=\{0,1,2,\cdots, T\}$.
For any Borel space $\mathcal{X}$, let $\Pc(\mathcal{X})$ represent the set of all probability measures on $\mathcal{X}$; as an example, $\Pc(\Dc)$ denotes the simplex of all probability measures over the space $\Dc$ and is given by 
\begin{align*}
	\Pc(\Dc) = \{ p = (p[1], \cdots, p[d]) \in \Rb_+^d: \, p[1] + \cdots + p[d]  = 1 \}.
\end{align*}
Hereafter, for a vector $p$, we use $p[x]$ to denote its $x$-th  component, and a similar notation $p[x, y]$ applies when $p$ is a matrix. For convenience, we denote $\Pc(\Dc^2)$ the set of all $d \times d$ matrices $p$ whose row vectors are in $\Pc(\Dc)$, 
\begin{align}
	\label{eq:tran}
	\Pc(\Dc^2) = \{ p = (p[x, y])_{x, y \in \Dc} : p[x, \cdot] \in \Pc(\Dc) \text{ for all } x \in \Dc \}.
\end{align}
Let $\Delta := \Dc \times \Pc(\Dc)$ denote the state-distribution product space.

We consider a Markov game of a large number of homogeneous agents, with the same preferences (reward function). Each agent makes its own decision, while interacting with the other agents, which in turn determines the transition probabilities of its state process. The goal of each agent is to find an optimal control that maximizes its reward. Given the homogeneity in the game, we naturally look for a \emph{stationary} (Nash) equilibrium under which all agents follow the same optimal control. To that end, we consider a \emph{representative} agent, referred to as Agent $i$, and call the remaining agents \emph{generic} agents or simply the population; we further assume that Agent $i$ uses control $\pi^i = \{ \pi^i_t \}_{t \in \Tb}$ but all generic agents apply a common control $\pi = \{ \pi_t \}_{t \in \Tb}$. The intuition is as follows: if $\pi^*$ is an equilibrium control and assume that all generic agents follow $\pi = \pi^*$, then the same $\pi^*$ should solve the individual optimization problem of Agent $i$ over all admissible $\pi^i$.

For the moment, assume there is no model uncertainty. In this case, all agents know their transition probabilities exactly, which are fully characterized by $P = \{P_t\}_{t \in \Tb}$, where the matrix $P_t: U \times \Pc(\Dc) \rightarrow \Pc(\Dc^2)$ is a \emph{deterministic} function of the control-distribution pair for every $t \in \Tb$. To be precise, for an agent in state $x$ and applying control $u$ ($= \pi_t$ or $\pi_t^i$) at time $t$, the probability that it transitions to state $y$ at time $t+1$ is given by $P_t(u, \mu_t)[x, y]$ for all $y \in \Dc$, where $\mu_t \in \Pc(\Dc)$ is the state distribution of the population at time $t \in \Tb$ (that is, $\mu_t[x]$ is the proportion of agents in state $x$ at time $t$). Let $\mu = \{\mu_t\}_{t \in \Tb}$ denote the distribution flow of the population and assume the initial distribution, $\mu_0 = v \in \Pc(\Dc)$, is given;  the evolution dynamics of $\mu$ is given by 
\begin{align*}
	\mu_{t+1}[y] = \sum_{x \in \Dc} \, \mu_t[x] \cdot P_t(\pi_t, \mu_t)[x, y], \quad t \in \Tb, \, y \in \Dc.
\end{align*}
Note that $\mu$ only depends on its initial condition $v$ and the common control $\pi$, but it is \emph{independent} of Agent $i$'s control $\pi^i$; we replace $\mu$ by $\mu^{\pi, v}$ if we need to emphasize its dependence on $\pi$ and $v$. The distribution flow $\mu$ evolves \emph{deterministically} since $P_t$ is deterministic for all $t$. For every agent, given its current state $x$ and control $u$, and the population distribution $\nu$ (at the same time point), let $r(x, u, \nu, y)$ denote the one-step, time-homogeneous reward function if it reaches state $y$ in the next time period, and we assume that $r$ is measurable and bounded.  When there is no model uncertainty, we define Agent $i$'s utility function (the total discounted reward) by 
\begin{align}
	\label{eq:V_no}
	V^{\pi^i, \pi}_{\mathrm{no}}(x,v) = \Eb \left[\sum_{t = 0}^{\infty} \rho^t  \, r \left(X^i_t, \pi^i_t, \mu_t,  X^i_{t+1} \right)  \bigg| X^i_0 = x, \mu_0 = v \right],
\end{align}
where $\rho \in (0,1)$ is a discounting factor, and $X^i = \{X_t^i\}_{t \in \Tb}$ denotes Agent $i$'s state process under $(\pi^i, \mu)$.

As argued in the previous subsection, it is unlikely in most applications to know the transition probabilities $P$ precisely, and this motivates us to consider uncertain transition probabilities belonging to some uncertainty set $\Pf$ in the above Markov game. However, such a robust extension is far from trivial, and it introduces at least two key challenges. 
\begin{itemize}
	\item 
    First, even for a fixed common control $\pi$ and an initial distribution $v \in \mathcal{P}(\mathcal{D})$, the distribution flow $\mu$ becomes \emph{stochastic} because the transition kernel $P$ is no longer deterministic but instead belongs to some uncertainty set $\Pf$. Recall that in the absence of model uncertainty, the deterministic nature of $\mu$ enables optimal strategies that depend solely on the current state and \emph{not} on the distribution (see \citet{saldi2018markov}); hence it suffices to consider state-dependent Markov controls of the form $u_t(x) \in U$ for all $t \in \mathcal{T}$. Evidently, the introduction of model uncertainty induces a fundamental change---it leads to a stochastic flow $\mu$, and the optimal decision at time $t$ may depend on both the current state $x$ and the \emph{realized} distribution $\nu :=\mu_t$. Therefore, we extend the class of Markov controls from purely state-dependent to state- and distribution-dependent policies $u(x, \nu)$. Moreover, we allow \emph{relaxed} controls, meaning that $u(x, \nu) \in \mathcal{P}(U)$ is a probability distribution over the compact action space $U$.
	
	\item Second, the ``robust'' extension of $V^{\pi^i, \pi}_{\mathrm{no}}$ in \eqref{eq:V_no} requires a delicate characterization of the ``worst-case" scenario to ensure both analytical tractability and economic interpretability. One natural approach is to employ the dynamic programming principle (DPP) and define Agent $i$'s robust utility function by 
	\begin{align}
		\label{eq:V_DPP}
			V^{\pi^i, \pi} (x, v) = \inf_{P \in \Pf} \; \Eb\left[ r \left(X_0^i,  \pi_0^i, \mu_0, X_1^i \right) + \rho \cdot V^{\pi^i, \pi}(X_1^i, \mu_1) \bigg| X_0^i = x, \, \mu_0 = v \right].
	\end{align}
    Recall that $V^{\pi^i, \pi}_{\mathrm{no}}$ in \eqref{eq:V_no} satisfies a similar equation \emph{without} the inf operator. The above definition of $V^{\pi^i, \pi}$ is heuristic, because we have not formally defined the uncertainty set $\Pf$ nor the exact mechanism of identifying the minimizer over all $P \in \Pf$.  This task is highly technical; we defer the details to the next section.
\end{itemize}

\subsection{Summary of Main Results and Contributions}

We summarize the main results of this paper and compare our work with related literature to outline the key contributions.

\begin{enumerate}
    \item \textbf{Extended strategy space accommodating distributional dependence.} We extend the strategy space to allow for simultaneous dependence on the individual state $x$ and population distribution $v$. This extension enables agents to dynamically adjust their strategies based on the \emph{observed} population distributions, providing a flexible mechanism to respond to the stochastic evolution of the population distributions under uncertainty. While state- and distribution-dependent strategies are relatively rare in the existing literature, they arise naturally in our setting: under model uncertainty, the population distribution evolves stochastically even when strategies are fixed, making real-time distributional information essential for optimal decision making. We note that two recent papers, \citet{bayraktar2024time} and \citet{hofer2025markov}, that also consider state- and distribution-dependent strategies; the former focuses on time inconsistency in MFGs, while the latter aims to establish connections between Markov perfect equilibria and the Nash-Lasry-Lions master equation, both with objectives different from ours.
    
    \item \textbf{Rigorous robust objective functions with dynamic programming characterization.}
    We define robust objective functions and equilibrium concepts based on robust optimization principles combined with dynamic programming ideas. Our framework allows adversarial selection of transition probabilities that may depend on specific state-distribution pairs $(x,v)$, capturing the economic intuition that different combinations of individual states and population distributions may yield different worst-case scenarios. Importantly, we prove that equilibria can be equivalently characterized through a \emph{one-shot optimality condition} (Proposition~\ref{prop:one_shot_eps}), which simplifies both theoretical analysis and computational complexity. 
    
    \citet{langner2024markov} also study MFGs under model uncertainty, but they treat transition probabilities as an \emph{endogenous} part of the equilibrium definition. In addition, their MFG equilibrium relies on a central planner's precommitment, which raises two concerns. First, the solution is likely time-inconsistent; second, the framework cannot capture the \emph{individual} interest of agents (see Example \ref{exm:com}). 
    In contrast, our framework treats transition probabilities as exogenously determined by objective but unknown system dynamics, and the agents solve their individual robust optimization problems and adopt a strategy under the worst-case scenarios. As a result, our formulation naturally yields a time-consistent solution, which adapts to the \emph{realized} environment over time. We provide a detailed comparison between our work and \cite{langner2024markov} in Section~\ref{sub:com}. 

    \item \textbf{Convergence results between finite-agent games and MFGs.} We establish the asymptotic relationship between $N$-agent games and MFGs under model uncertainty, with the following two key results:
    \begin{itemize}
        \item \textbf{Approximation} (Theorem~\ref{thm:mf-limit}): Any mean-field equilibrium in the space of continuous strategies constitutes an $\varepsilon$-Nash equilibrium for sufficiently large finite populations, providing theoretical justification for using mean-field equilibria as practical solutions for large-scale games.
        
        \item \textbf{Convergence} (Theorem~\ref{thm:convergence-mf}): Equilibria of finite-agent games converge to equilibria of MFGs, as the number of agents increases to infinity.
    \end{itemize}
    These results fully justify the MFG framework as both an approximation tool and a limiting description for large but finite-agent games under model uncertainty.
    
    \item \textbf{Existence of equilibria for finite-agent games.} We establish the existence of stationary equilibria for $N$-agent games (Theorem~\ref{thm:n-agent-existence}) using Kakutani's fixed-point theorem. Combined with the convergence result, this provides an indirect pathway to the existence of MFG equilibria: any convergent sequence of $N$-agent equilibria yields a mean-field equilibrium. To demonstrate that MFG equilibria (Definition~\ref{def:eq}) can indeed exist, we construct a solvable two-state MFG example (Example~\ref{exm:MFG}) and derive the equilibrium strategy, worst-case transition kernel, and equilibrium value function, all in closed form. 
\end{enumerate}

The remainder of this paper is structured as follows. 
Section~\ref{sec:inf_MFG} introduces the mathematical framework for both mean-field and $N$-agent games under model uncertainty, with the main convergence and existence results in Subsection~\ref{sub:main}. Section~\ref{sec:exm} provides an example of MFG with closed-form solutions and conducts a detailed comparison with \citet{langner2024markov}. Section~\ref{sec:proof} collects the proofs of all main theorems. Section~\ref{sec:conclusion} concludes the paper. Appendix~\ref{app:ass} provides additional technical results on the assumptions imposed on the uncertainty set.

\section{MFG and $N$-Agent Game under Model Uncertainty}
\label{sec:inf_MFG}

\subsection{MFG Setup}
\label{sub:MFG}

In this subsection, we formally introduce the infinite-horizon MFG under model uncertainty, with $\Tb =\{0,1,2,\cdots\}$. Recall that under this MFG setup, a representative agent (called Agent $i$) applies a strategy $\pi^i$, while all other agents in the population apply a common strategy $\pi$. In Subsection \ref{sub:framework}, we define the representative agent's (Agent $i$'s) utility function $V^{\pi^i, \pi}_{\mathrm{no}}$, \emph{without} model uncertainty, by \eqref{eq:V_no}. When model uncertainty is incorporated, we follow the ``worst-case'' approach and define  Agent $i$'s \emph{robust utility function} $V^{\pi^i, \pi}$ intuitively as the infimum over some uncertainty set of transition probabilities. We first provide a formal definition of $V^{\pi^i, \pi}$.

To start, we introduce a set-valued mapping $\Ff: (x, u, v) \in \Dc \times U \times \Pc(\Dc) \mapsto \Ff(x, u, v) \subseteq \Pc(\Dc)$, which captures distributional uncertainty in the next period, given the current information $(x, u, v)$. Specifically, each element of $\Ff(x, u, v)$ represents a possible transition distribution for an agent whose current state is $x \in \Dc$ and whose control value is $u \in U$, when the state distribution of the population is $v \in \Pc(\Dc)$. 

We next define a function $\Pf$, which maps from $\Pc(\Dc)$ (the space of all population distributions $v$) to the product space $\prod_{(x, u) \in \Dc \times U} \Pc(\Dc)$, by 
\begin{align}
	\label{eq:Pf}
	\Pf(v) = 
	\left\{ \bP = \{ p(x,u) \}_{(x,u) \in \Dc \times U} \,\Big|\, 
	\begin{array}{l}
		p(x,u) \in \Ff(x, u, v), \,\forall \, (x,u) \in \Dc \times U, \\[3pt]
		p(\cdot, \cdot): \Dc \times U \to \Pc(\Dc) \text{ is  $L$-Lipschitz continuous}
	\end{array}
	\right\},
\end{align}
where the $L$-Lipschitz continuity condition requires (for some $L > 0$) 
\begin{align}
\label{eq:F_Lip}
	d_{W_1} \big( p(x_1, u_1), p(x_2, u_2)  \big) \le L \big(d_\Dc(x_1, x_2) + |u_1 - u_2|\big)
\end{align}
for all $(x_1,u_1), (x_2,u_2) \in \Dc \times U$.\footnote{$d_{W_1}$ denotes the Wasserstein-1 distance on $\Pc(\Dc)$, and $d_\Dc$ denotes the discrete metric over the state space $\Dc$.
The constant $L$ in \eqref{eq:F_Lip} is a model parameter that reflects the degree of regularity imposed on transition kernels. In applications where the uncertainty set $\Ff$ admits a Wasserstein ball structure, a natural choice is to set $L$ equal to the Lipschitz constant of the ball's center function (see Appendix~\ref{app:ass}).}

To understand the definition of $\Pf$ in \eqref{eq:Pf}, note that at any time, the population distribution $v$ is unique, but agents may be in any state $x \in \Dc$, and the realized (or sampled) value of their control may take any value $u \in U$ (because they are allowed to apply \emph{relaxed} controls).
This explains why the uncertainty set includes all possible transition probabilities over $(x, u) \in \Dc \times U$, while keeping $v \in \Pc(\Dc)$ fixed.
Therefore, every element $\bP$ in $\Pf(v)$ provides a complete characterization of the transition for all agents when the current distribution of the population is $v$, and we refer to each $\bP$ as a (full) \emph{transition kernel}. As time $t$ evolves, so does the population distribution $\mu_t$, and the uncertainty set $\Pf(\mu_t)$ clearly depends on $\mu_t$; to emphasize this dependence, we often write the element in $\Pf(\mu_t)$ as $\bP(\mu_t) = \{ p(x,u, \mu_t) \}_{(x,u) \in \Dc \times U}$.

We now follow an induction approach to formally characterize the uncertainty set at every time $t \in \Tb$ and specify the probability measure induced by its elements.
At time 0, suppose that the initial state distribution of the population is given as $\mu_0 \in \Pc(\Dc)$; the corresponding uncertainty set $\Pf_0$ is given by $\Pf_0 = \Pf(\mu_0)$, where $\Pf(\cdot)$ is defined in \eqref{eq:Pf}. For every $\bP_0 := \bP_0(\mu_0) = \{p_0(x, u, \mu_0)\}_{(x, u) \in \Dc \times U} \in \Pf_0$, it induces a probability measure, $\Pb^{\bP_0}$, under which the state transition of Agent $i$ (who applies a relaxed control $\pi^i$) from time 0 to time 1 is governed by 
\begin{align}
	\label{eq:X1}
	\Pb^{\bP_0}(X_{1}^i &= y | X_0^i = x) = \int_U \, p_0(x, u, \mu_0)[y] \, \pi_0^i (\dd u), \quad x, y \in \Dc, 
\end{align}
and the population distribution at time 1 is given by
\begin{align}
	\label{eq:mu1}
	\mu_{1} [y] = \sum_{x \in \Dc} \int_U \, \mu_0[x] \cdot p_0(x, u, \mu_0)[y] \, \pi_0(\dd u),  \quad y \in \Dc.
\end{align}

With $\mu_1$ obtained in \eqref{eq:mu1}, we follow the same procedure to define the uncertainty set at time 1 by $\Pf_1 := \Pf(\mu_1)$. Then, every $\bP_1 := \bP_1(\mu_1) = \{p_1(x, u, \mu_1)\}_{(x, u) \in \Dc \times U} \in \Pf_1$ induces a probability measure, $\Pb^{\bP_1}$, such that the transition of Agent $i$ and the population distribution $\mu_2$ satisfy similar equations as \eqref{eq:X1} and \eqref{eq:mu1}, respectively, but with $p_0(\cdot, \cdot, \mu_0)$ replaced by $p_1(\cdot, \cdot, \mu_1)$. Proceeding recursively, we define $\Pf_t := \Pf(\mu_t)$ and $\Pb^{\bP_t}$ for all $t \in \Tb$. Then, we define the uncertainty set $\overline{\Pf}^{\infty}$ as the ``aggregate'' of all $\Pf_t$ over the infinite time horizon $\Tb$, 
\begin{align*}
	\overline{\Pf}^{\infty} := \Pf_0 \otimes \Pf_1 \otimes \cdots 
\end{align*}
and for every $\overline{\bP}^{\infty} = \{ \bP_t \}_{t \in \Tb} \in \overline{\Pf}^{\infty}$, we define a probability measure $\Pb^{\overline{\bP}^{\infty}}$ as the product measure of all $\Pb^{\bP_t}$, under which the state transition of Agent $i$ satisfies 
\begin{align}
	\label{eq:X}
	\Pb^{\overline{\bP}^{\infty}} (X_{t+1}^i = y | X_t^i = x) = \int_U \, p_t(x, u, \mu_t)[y] \, \pi_t^i (\dd u), 
\end{align} 
and the population distribution evolves according to
\begin{align}
	\label{eq:mu}
	\mu_{t+1} [y] = \sum_{x \in \Dc} \int_U \, \mu_t[x] \cdot p_t(x, u, \mu_t)[y] \, \pi_t(\dd u)
\end{align} 
for all $x, y \in \Dc$ and $t \in \Tb$, where the initial conditions, $X_0^i \in \Dc$ and $\mu_0 \in \Pc(\Dc)$, are fixed.

With the above preparation, we define Agent $i$'s robust utility function $V^{\pi^i, \pi}$ by
\begin{align}
	\label{eq:V_inf}
	V^{\pi^i, \pi} (x, v) &= \inf_{\overline{\bP}^{\infty} \in \overline{\Pf}^{\infty}} \; \Eb^{\overline{\bP}^{\infty}}_{x, v} \left[ \sum_{t=0}^{\infty} \rho^t \cdot r \left(X_t^i,  \pi_t^i,  \mu_t, X_{t+1}^i \right)\right] 
\end{align}
for all $(x, v) \in \Delta = \Dc \times \Pc(\Dc)$, 
where $\Eb^{\overline{\bP}^{\infty}}_{x, v}$ denotes the conditional expectation under $\Pb^{\overline{\bP}^{\infty}}$ given $X_0^i = x$ and $\mu_0 = v$, $X^i$ and $\mu$ evolve according to \eqref{eq:X} and \eqref{eq:mu}, respectively, and $\rho \in (0,1)$ is the discounting factor. We impose the following assumption on the (one-period) homogeneous reward function $r$.

\begin{assumption}
	\label{ass:r}
	The reward function $r: (x, u, v, y) \in \Dc \times U\times \mathcal{P}(\Dc)\times \Dc \mapsto r(x, u, v, y)\in \Rb$ is bounded, and Lipschitz continuous in the arguments $ v \in \Pc(\Dc)$ and $u\in U$, 
	in the sense that there exist constants $C_r > 0$ and $L_r > 0$ such that for every 
	$x, y \in \Dc$, $u_1, u_2 \in U$, and $v_1, v_2 \in \Pc(\Dc)$,
	\begin{align*}
		|r(x, u, v, y)| \leq C_r \quad \text{and} \quad |r(x,u_1, v_1, y) - r(x, u_2, v_2, y)| 
		\leq L_r \cdot ( d_{W_1}(v_1, v_2) +|u_1-u_2|).
	\end{align*}
\end{assumption}

For tractability, we also need to impose regularity assumptions on the uncertainty mapping $\Pf$ in \eqref{eq:Pf}. To this end, we first equip the range space of $\Pf$---the product space $\prod_{(x, u) \in \Dc \times U} \Pc(\Dc)$---with the supremum metric: for all $\bP = \{p(x, u)\}_{(x, u) \in \Dc \times U}$ and $\mathbf{Q} = \{q(x, u)\}_{(x, u) \in \Dc \times U}$
 in this space,
\begin{align*}
	d_{\infty}(\bP, \mathbf{Q}) = \sup_{(x, u) \in \Dc \times U} \, d_{W_1} \left(p(x, u), q(x, u)\right).
\end{align*}
We state the assumptions on $\Pf$ below. 

\begin{assumption}
	\label{ass:uncertainty}
	The mapping $\Pf: v \in \Pc(\Dc) \mapsto \Pf (v) \subseteq \prod_{(x, u) \in \Dc \times U} \Pc(\Dc)$ is non-empty, convex-valued, 
	compact-valued, and continuous (in the Hausdorff metric induced by $d_\infty$). 
\end{assumption}

Assumption \ref{ass:uncertainty} plays a key role in the subsequent analysis because it provides the topological structure necessary for establishing
the existence of equilibria via a fixed-point argument. Recall that $\Pf$ is defined via $\Ff$. For technical convenience, Assumption~\ref{ass:uncertainty} imposes conditions on $\Pf$ rather than directly on $\Ff$. Assumption~\ref{ass:uncertainty} is satisfied under a broad range of settings in the robust optimization literature. Two important cases are highlighted below.
\begin{itemize}
	\item \emph{Wasserstein ball structure.} 
    When $\Ff$ admits a Wasserstein ball structure, as commonly studied in distributionally robust optimization (see \citet{mohajerin2018data} and \citet{gao2023distributionally}), the required properties follow from Proposition~\ref{prop:mapping} in Appendix~\ref{app:ass}.
	
	\item \emph{Distribution-independent uncertainty.} 
    When the uncertainty set does not depend on the population distribution $v$, as typically assumed in the robust MDP literature (see \citet{iyengar2005robust} and \citet{nilim2005robust}), Assumption~\ref{ass:uncertainty} reduces to standard regularity conditions on $\Ff$ that are straightforward to verify.
\end{itemize}

Next, we define an equilibrium to the infinite-horizon MFG under model uncertainty. To that end, we first discuss the set of admissible strategies. 
Recall from the standard framework in Subsection \ref{sub:framework} that agents' strategies are Markovian,  depending on the state-distribution pair, and take values in the space of $\Pc(U)$ (that is, they are relaxed controls). We thus define the space $\Pi$ as follows:
\begin{align}
	\label{eq:Pi}
	\Pi = \{ \phi : \phi \text{ is Borel measurable, mapping from  $(x, v) \in \Delta$ to $\Pc(U)$} \}.
\end{align}
Given the observed state-distribution pair $(X_t^i, \mu_t)$, each $\phi$ in $\Pi$ induces a relaxed control $\phi(X_t^i, \mu_t)$  for Agent $i$ at time $t$. 
\begin{definition}
	\label{def:ad_MFG}
	A strategy $\pi^i = \{\pi^i_t\}_{t \in \Tb}$ is called \emph{admissible} for Agent $i$ if for every $t \in \Tb$, there exists a mapping $\boldsymbol{\pi^i_t} \in \Pi$ such that $\pi_t^i = \boldsymbol{\pi^i_t}(X_t^i, \mu_t)$,  where $X_t^i$ is Agent $i$'s state governed by \eqref{eq:X}, and $\mu_t$ is the population distribution as in \eqref{eq:mu}, both at time $t$. 	
	An admissible strategy $\pi^i = \{\pi^i_t\}_{t \in \Tb}$ is called \emph{time-homogeneous} if the mappings $\boldsymbol{\pi^i_t}$ are the same for all $t \in \Tb$; that is, there exists one $\boldsymbol{\pi^i} \in \Pi$ such that $\pi^i_t = \boldsymbol{\pi^i} (X_t^i, \mu_t)$ for all $t \in \Tb$. 
\end{definition}

In Definition \ref{def:ad_MFG}, we implicitly assume that all other agents apply a common strategy $\pi$ under which the flow of population distributions $\mu = \{\mu_t\}_{t \in \Tb}$ evolves by \eqref{eq:mu}, with $\mu_0 \in \Pc(\Dc)$ given. Because both $X_t^i$ and $\mu_t$ are observable at time $t$, Agent $i$'s strategy at $t$, $\pi_t^i$, is uniquely determined by a mapping $\boldsymbol{\pi^i_t} \in \Pi$. Given this one-to-one correspondence, we use $\pi^i_t$ (without bold font) to denote both a mapping in $\Pi$ and the admissible strategy it generates.
With this ``abuse'' of notation, for an admissible strategy $\pi^i = \{\pi^i_t\}_{t \in \Tb}$, we have $\pi^i_t \in \Pi$ for all $t \in \Tb$. For convenience, we introduce the following shorthand notation:
\begin{align*}
	\pi^i \text{ is admissible } \Leftrightarrow \pi^i \in \Pi^{\Tb} \quad \text{and} \quad 
	\pi^i \text{ is admissible and time-homogeneous} \Leftrightarrow \pi^i \in \Pi.
\end{align*}

The MFG we consider here is \emph{homogeneous} in two aspects. 
First, although time-dependent strategies $\pi^i = \{\pi^i_t\}_{t \in \Tb} \in \Pi^{\Tb}$ are allowed, the infinite horizon and time-homogeneity of the objective and dynamics suggest that an equilibrium strategy $\pi^{i,*}$ for Agent $i$ should be \emph{time-homogeneous} (i.e., $\pi^{i,*} \in \Pi$). 
Second, because all agents are ``\emph{identical}'', sharing the same objective in \eqref{eq:V_inf} and transition dynamics in  \eqref{eq:X}, it is reasonable to conjecture that under equilibrium, all agents should apply the \emph{same} control (that is, $\pi^{i,*}  \equiv \pi^*$ for all $i$).
This leads to the following definition of \emph{stationary} (Nash) equilibrium for the infinite-horizon MFG.	

\begin{definition}
	\label{def:eq}
	Suppose that Agent $i$ follows an admissible strategy $\pi^i \in \Pi^{\Tb}$, while all other agents apply a common admissible, time-homogeneous strategy $\pi^* \in \Pi$. $\pi^*$ is called a stationary equilibrium if it satisfies the following full optimality condition: 
	\begin{align}
		\label{eq:pi_op}
		V^{\pi^*,\pi^*}(x, v) = \sup_{\pi^i \in \Pi^{\Tb}} \, V^{\pi^i, \pi^*}(x, v), \quad \text{for all } (x, v) \in \Delta,
	\end{align}
	where $V^{\pi^i, \pi}$ is defined by \eqref{eq:V_inf}.	
	For $\varepsilon \geq 0$, we call $\pi^{*,\varepsilon}$ a stationary $\varepsilon$-equilibrium if \eqref{eq:pi_op} is replaced by 
	\begin{align}
		\label{eq:pi_eps}
		V^{\pi^{*,\varepsilon}, \pi^{*,\varepsilon}}(x, v) \ge \sup_{\pi^i \in \Pi^{\Tb}} \, V^{\pi^i, \pi^{*,\varepsilon}}(x, v) - \varepsilon.
	\end{align} 
\end{definition}

In Definition \ref{def:eq}, we also define $\varepsilon$-equilibrium $\pi^{*,\varepsilon}$ by \eqref{eq:pi_eps} for computational advantage and for handling the possible cases where an exact equilibrium may not exist.

\begin{remark}[Finite horizon extension]
Although this paper focuses on the infinite-horizon setting, our framework can be naturally extended to finite-horizon problems with a terminal time 
$T \in \mathbb{N}$. The main adjustments required are:
(i) replacing stationary strategies with time-dependent policies 
$(\pi_0, \ldots, \pi_{T-1})$;
(ii) incorporating a terminal reward function 
$g: \Delta \to \mathbb{R}$; and
(iii) using backward induction in place of fixed-point characterizations.
The convergence results (Theorems \ref{thm:mf-limit} and \ref{thm:convergence-mf}) and the existence result (Theorem \ref{thm:n-agent-existence}) remain valid in this setting, with proofs following analogous lines. Indeed, the comparison example in Subsection~\ref{sub:com} is formulated and solved in a finite-horizon setting.

Our focus on the infinite-horizon case stems from its analytical advantages: it allows us to work with stationary equilibria, which provide a natural and often more tractable solution concept for games without a predetermined endpoint.
\end{remark}

\subsection{Characterization of MFG Equilibrium}
\label{sub:MFG_equ}

In this subsection, we aim to obtain an analytical characterization of the stationary equilibrium to the MFG defined by \eqref{eq:pi_op}. The first challenge we face here is that the robust utility function in \eqref{eq:V_inf}  involves an infinite-dimensional optimization problem over $\overline{\Pf}^{\infty} = \otimes_{t = 0}^{\infty} \, \Pf_t$, and this definition does not yield a computational method for determining the ``worst scenario'' in  $\overline{\Pf}^{\infty}$. To address this, we show that $V^{\pi^i, \pi}$ in \eqref{eq:V_inf} can be equivalently characterized by the following DPP:
\begin{align}
	\label{eq:V}
	V^{\pi^i, \pi} (x, v) &= \inf_{\bP_0 \in \Pf_0} \; \Eb^{\bP_0}_{x, v} \left[  r \left(X_0^i,  \pi_0^i,  \mu_0, X_1^i \right) + \rho \cdot V^{\pi^i, \pi}(X_1^i, \mu_1)\right], 
\end{align}
where the subscript in $\Eb^{\bP_0}_{x, v}$ denotes the initial conditions $ X_0^i = x \in \Dc$ and $\mu_0 = v \in \Pc(\Dc)$, and the uncertainty set is $\Pf_0 = \Pf(v)$. Note that for $V^{\pi^i, \pi}(X_1^i, \mu_1)$ on the right-hand side of \eqref{eq:V}, the strategies are the restricted versions of $\pi^i$ and $\pi$ over $t \in \{1,2,\cdots\}$, and we write $V^{\pi^i, \pi}(X_1^i, \mu_1)$, instead of the more accurate notation $V^{ \{\pi^i_t\}_{t=1,2,\cdots}, \{\pi_t\}_{t=1,2,\cdots}}(X_1^i, \mu_1)$, for simplicity.

While the above characterization of $V^{\pi^i, \pi}$ is intuitive, we have not yet established the existence and uniqueness of a solution to the DPP equation in \eqref{eq:V}. The result below confirms this conjectured equivalence.

\begin{proposition}
	\label{prop:V_same}
	For all admissible strategies $\pi^i, \pi \in \Pi^{\Tb}$ and all $(x,v) \in \Delta$, the robust utility function $V^{\pi^i, \pi}$ in \eqref{eq:V_inf} satisfies the DPP in \eqref{eq:V}.  
	Moreover, $V^{\pi^i, \pi}$ in \eqref{eq:V_inf} is the unique bounded function on $\Pi^{\mathbb{T}} \times \Pi^{\mathbb{T}} \times \Delta$ to the DPP in \eqref{eq:V}.
\end{proposition}

\begin{proof}
	The claim that $V^{\pi^i, \pi}$ in \eqref{eq:V_inf} satisfies the DPP in \eqref{eq:V} follows directly from the definition, which also establishes the existence of a solution to \eqref{eq:V}. Thus, it remains to show uniqueness. To that end, assume to the contrary that there exists another bounded function, $W: \Pi^{\mathbb{T}} \times \Pi^{\mathbb{T}} \times \Delta \to \mathbb{R}$,  solving \eqref{eq:V}. Define the distance between the two solutions by
	\begin{align*}
		d(V, W) = \sup_{\pi^i \in \Pi^{\mathbb{T}}, \pi \in \Pi^{\mathbb{T}}, (x,v) \in \Delta}  \big| V^{\pi^i,\pi}(x,v) - W^{\pi^i,\pi}(x,v) \big|.
	\end{align*}
    For every fixed $(\pi^i, \pi, x, v)$, using the DPP and the standard inf-sup inequality yields 
    \begin{align*}
		|V^{\pi^i,\pi}(x,v) - W^{\pi^i,\pi}(x,v)| 
		&\leq \sup_{\bP \in \Pf(v)} \mathbb{E}^{\bP} \left[\rho \left|V^{\pi^i_{1:\infty},\pi_{1:\infty}}(X_1^i, \mu_1) - W^{\pi^i_{1:\infty},\pi_{1:\infty}}(X_1^i,\mu_1)\right|\right] \\
		&\leq \rho \, \sup_{X_1^i, \mu_1} \left|V^{\pi^i_{1:\infty},\pi_{1:\infty}}(X_1^i, \mu_1) - W^{\pi^i_{1:\infty},\pi_{1:\infty}}(X_1^i,\mu_1)\right|,
	\end{align*}
    where $\pi^i_{1:\infty}$ ($\pi_{1:\infty}$) denotes the restriction of $\pi^i$ ($\pi$) over all $t = 1, 2, \cdots$. Because $(\pi^i, \pi)$ ranges over $\Pi^{\mathbb{T}} \times \Pi^{\mathbb{T}}$, so does their tail strategies $(\pi^i_{1:\infty}, \pi_{1:\infty})$. Therefore, 
    \begin{align*}
    	\sup_{X_1^i, \mu_1} \left|V^{\pi^i_{1:\infty},\pi_{1:\infty}}(X_1^i, \mu_1) - W^{\pi^i_{1:\infty},\pi_{1:\infty}}(X_1^i,\mu_1)\right| \le d(V, W).
    \end{align*}
    Finally, taking supremum over all $(\pi^i, \pi, x, v)$ yields $d(V, W) \leq \rho \cdot d(V, W)$. Because $\rho \in (0,1)$, $d(V, W) = 0$, implying $V \equiv W$. 
\end{proof}

Thanks to Proposition \ref{prop:V_same}, the two definitions of $V^{\pi^i, \pi}$ in \eqref{eq:V_inf} and \eqref{eq:V} are equivalent. However, the second one in \eqref{eq:V}  possesses two significant advantages over \eqref{eq:V_inf}. The first advantage is that the optimization problem in \eqref{eq:V} is solved over $\Pf_0$, whereas the one in \eqref{eq:V_inf} is solved over $\overline{\Pf}^{\infty} = \Pf_0 \otimes \Pf_1 \otimes \Pf_2 \otimes \cdots$. Second, the DPP in \eqref{eq:V}  naturally yields an iterative algorithm that can be used to compute the robust utility function $V^{\pi^i, \pi}$ as the fixed point of the Bellman operator $\mathcal{T}^{\pi^i, \pi}$. Therefore, we mainly work with the DPP-based definition in \eqref{eq:V} hereafter, due to its tractability.

We proceed to discuss how we seek a stationary equilibrium $\pi^* \in \Pi$ in Definition \ref{def:eq}. 
However, obtaining $\pi^*$ directly from \eqref{eq:pi_op} is challenging: we lack an \emph{a priori} ansatz for $\pi^*$, and directly solving the supremum problem appears intractable.
Indeed, we would first solve $\sup\limits_{\pi^i} \, V^{\pi^i, \pi}(x, v)$ for every fixed $\pi$ (adopted by all other agents) and denote its solution by $\widehat{\pi}^i$; next, we treat $\widehat{\pi}^i := \widehat{\pi}^i(\pi)$  as a functional of $\pi$ and obtain $\pi^*$ as the fixed point of the maximizing functional, $\pi^* =  \widehat{\pi}^i(\pi^*)$. Despite this intuition, we do not pursue this approach directly.

Before introducing our approach, recall from Definition~\ref{def:ad_MFG} that an admissible strategy $\pi^i$ takes the precise form $\pi^i_t = \pi^i_t(X^i_t, \mu_t)$ for all $t \in \Tb$. The equilibrium condition in \eqref{eq:pi_op} is termed the ``full optimality condition'' because Agent~$i$ is allowed to deviate from $\pi^*$ at \emph{all} time points $t \in \Tb$. Intuitively, however, it is not necessary to test deviations at every instant; due to the time-homogeneous nature of the mean-field game, only a deviation in the initial period needs to be considered. This observation motivates the following notion of a perturbed strategy. For any admissible strategy $\pi \in \Pi$, we define a \emph{one-shot perturbation} of $\pi$ by
\begin{align}
	\label{eq:pi_pert}
	\phi \otimes_1 \pi = \begin{cases}
		\phi(X_0^i, \mu_0), & t = 0 ,\\
		\pi(X_t^i, \mu_t), & t = 1, 2, \cdots
	\end{cases}
\end{align}
for some $\phi \in \Pi$. The perturbed strategy $\phi \otimes_1 \pi$ differs from $\pi$ only in the first period, and they are identical in all other periods. Assuming that $\pi^* \in \Pi$ is an equilibrium strategy, and that Agent $i$ applies $\pi^i \otimes_1 \pi^*$ for some $\pi^i \in \Pi$, we naturally conjecture that optimizing Agent $i$'s robust utility over all one-shot perturbed strategies should return a solution equal to $\pi^*$. Below we formally prove a stronger version of this conjecture for $\varepsilon$-equilibrium $\pi^{*,\varepsilon}$ defined by \eqref{eq:pi_eps}.

\begin{proposition}
	\label{prop:one_shot_eps}
	Let $\varepsilon\geq0$ and $\pi^{*,\varepsilon} \in \Pi$ be an admissible, time-homogeneous strategy. Then, the following results hold:
	\begin{enumerate}
		\item[(i)] If $\pi^{*,\varepsilon}$ is a stationary $\varepsilon$-equilibrium, then it satisfies the $\varepsilon$-one-shot optimality condition:
		\begin{align}
			\label{eq:pi_op_one_eps_nec}
			V^{\pi^{*,\varepsilon},\pi^{*,\varepsilon}}(x, v) \ge \sup_{\pi^i \in \Pi} \; V^{\pi^i \otimes_1 \pi^{*,\varepsilon}, \pi^{*,\varepsilon}}(x, v) - \varepsilon, \quad \text{for all } (x, v) \in \Delta.
		\end{align}
		\item[(ii)] Conversely, if $\pi^{*,\varepsilon}$ satisfies \eqref{eq:pi_op_one_eps_nec}, then it is a stationary $\frac{2\varepsilon}{1 - \rho}$-equilibrium.
	\end{enumerate}
\end{proposition}

\begin{proof}
	\textbf{Part (i).} Because one-shot deviations are special cases of arbitrary deviations, the $\varepsilon$-equilibrium condition immediately implies the $\varepsilon$-one-shot optimality condition.
	
	\textbf{Part (ii).} Suppose that $\pi^{*,\varepsilon}$ satisfies the $\varepsilon$-one-shot optimality condition. We first prove by induction that for all $k$-period perturbed strategies $\pi^i_{0:k-1} \otimes_{k} \pi^{*,\varepsilon}$ (defined similarly as in \eqref{eq:pi_pert}),
	\begin{equation}
		\label{eq:finite-horizon-bound}
		V^{\pi^i_{0:k-1} \otimes_{k} \pi^{*,\varepsilon}, \pi^{*,\varepsilon}}(x,v) \leq V^{\pi^{*,\varepsilon}, \pi^{*,\varepsilon}}(x,v) + \varepsilon \sum_{j=0}^{k-1} \rho^j, \quad \forall (x,v) \in \Delta.
	\end{equation}
	The base case $k=1$ is exactly the $\varepsilon$-one-shot optimality condition. For the inductive step, the DPP in \eqref{eq:V} combined with the induction hypothesis yields
	\begin{align*}
		V^{\pi^i_{0:k} \otimes_{k+1} \pi^{*,\varepsilon},\pi^{*,\varepsilon}}(x,v) 
		&\leq  \inf_{\bP \in \Pf(v)} \mathbb{E}_{x, v}^{\bP} \left[r(x, \pi_0^i, v, X_1^i) + \rho \Big( V^{\pi^{*,\varepsilon}, \pi^{*,\varepsilon}}(X_1^i, \mu_1) + \varepsilon \sum_{j=0}^{k-1} \rho^j \Big) \right] \\
		&\leq V^{\pi^{*,\varepsilon}, \pi^{*,\varepsilon}}(x,v) + \varepsilon \sum_{j=0}^{k} \rho^j.
	\end{align*}
	
	For arbitrary $\pi^i \in \Pi^{\Tb}$, the difference in utility from periods $k$ onward is bounded by $\frac{2C_r \rho^k}{1-\rho}$, where $C_r$ is the bound on $|r|$ from Assumption~\ref{ass:r}.  Choosing $k$ such that $\frac{2C_r \rho^k}{1-\rho} \leq \varepsilon$, we obtain
	\begin{align*}
		V^{\pi^i, \pi^{*,\varepsilon}}(x,v) \le V^{\pi^i_{0:k-1} \otimes_{k} \pi^{*,\varepsilon}, \pi^{*,\varepsilon}}(x,v) + \varepsilon \le V^{\pi^{*,\varepsilon}, \pi^{*,\varepsilon}}(x,v) + \frac{2 \varepsilon}{1 - \rho}. 
	\end{align*}
	Taking supremum over $\pi^i \in \Pi^{\Tb}$ completes the proof.
\end{proof}

\begin{remark}
	As an immediate consequence of Proposition \ref{prop:one_shot_eps}, $\pi^{*} \in \Pi$ is a stationary equilibrium if and only if it satisfies the one-shot optimality condition:
	\begin{align}
		\label{eq:pi_op_one}
		V^{\pi^{*},\pi^{*}}(x, v) = \sup_{\pi^i \in \Pi} \; V^{\pi^i \otimes_1 \pi^{*},\pi^{*}}(x, v), \quad \text{for all } (x, v) \in \Delta.
	\end{align}
	
	For approximate equilibria, Proposition \ref{prop:one_shot_eps} shows that verifying the one-shot optimality condition with error $\varepsilon$ is sufficient to establish a full $\frac{2\varepsilon}{1-\rho}$-equilibrium. This characterization is essential for our convergence results in Subsection~\ref{sub:main}: when establishing that mean-field equilibria provide approximate solutions to $N$-agent games, we can restrict our verification to the simpler one-shot condition rather than checking optimality over all deviations in $\Pi^{\Tb}$.
\end{remark}

\subsection{$N$-Agent Game}
\label{sub:N_agent}

MFG models provide a powerful framework for analyzing large populations, yet their practical relevance hinges on how well MFG equilibria approximate those of finite‑player games. To bridge the MFG formulation under model uncertainty in  Subsection~\ref{sub:MFG} with the finite‑population setting, we now introduce the corresponding $N$‑agent game. This formulation serves as the basis for establishing the approximation result in the next subsection. Throughout, $N$ denotes a positive integer.

Consider an infinite-horizon Markov game with $N$ agents, indexed by $j \in \Nc := \{1, 2, \cdots, N \}$, and all agents share the same reward function $r$ satisfying Assumption \ref{ass:r}. 
We denote Agent $j$'s control and state by $\pi^{N, j} = \{ \pi_t^{N, j}\}_{t \in \Tb}$ and $X^{N, j} = \{ X_t^{N, j}\}_{t \in \Tb}$, respectively, with the initial state $X_0^{N, j} = x^j \in \Dc$ given, for all $j \in \Nc$. Recall that the state space is $\Dc = \{1, 2, \cdots, d\}$ (implying $X_t^{N, j} \in \Dc$ for all $t$ and $j$), and controls are distributions in the action space $U$ (that is, $\pi_t^{N, j} \in \Pc(U)$). At the ``current'' time $t$, all agents' states, $\vec{X}^N_t := \{X_t^{N,1}, \ldots,X_t^{N,N} \}$, are observable, and we define the \emph{empirical distribution} of Markov game at $t$, $\mu_t^N \in \Pc(\Dc)$, by 
\begin{align}
	\label{eq:mu_N}
	\mu_t^N [x] = \frac{1}{N}\sum_{j=1}^N \delta_{X_t^{N,j}}(x), \quad x \in \Dc, 
\end{align} 
where $\delta_X (x) = 1$ if $X = x$ and $\delta_X (x) = 0$ otherwise. Since $\mu_t^N$ is an empirical distribution, $N \mu_t^N[x] \in \Nb$ for all $x \in \Dc$.
By \eqref{eq:mu_N}, if $\mu_t^N [x] > 0$ for some $x \in \Dc$, then at least one of the $N$ agents is in state $x$ at time $t$; in this case, we call the pair $(x, \mu_t^N)$ \emph{feasible}. With these two conditions, we define the set of all feasible state-distribution pairs by 
\begin{align*}
	\Delta_N := \{(x,v) \in \Delta :  N v \in \Nb^d \text{ and } N v[x] > 0 \} .
\end{align*}
Note that $\Delta_N$ is finite and $\bigcup_{N=1}^{\infty} \Delta_N$ is dense in $\Delta$. Correspondingly, we define the space of admissible strategies for the $N$-agent game by
\begin{align}
	\label{eq:Pi_N}
	\Pi_N = \{ \phi : \phi \text{ is Borel measurable, mapping from  $(x, v) \in \Delta_N$ to $\Pc(U)$} \}.
\end{align}
Note that any $\phi \in \Pi$ naturally restricts to an element of $\Pi_N$.

To set up the $N$-agent game under model uncertainty, we follow the MFG model in Subsection \ref{sub:MFG} but replace the population distribution $\mu = \{\mu_t\}_{t \in \Tb}$ in the MFG by the empirical distribution $\mu^N =\{\mu_t^N\}_{t \in \Tb}$ in \eqref{eq:mu_N}. At every $t \in \Tb$, $\vec{X}^N_t$ and $\mu_t^N$ are known, and we define the corresponding uncertainty set by $\Pf(\mu_t^N)$, where the functional $\Pf$ is defined by \eqref{eq:Pf}. 
Similarly, for every transition kernel $\bP_t^N = \{p(x, u, \mu_t^N)\}_{(x, u) \in \Dc \times U}$ in $\Pf(\mu_t^N)$, it induces a probability measure $\Pb^{\bP_t^N}$; for all $\overline{\bP}^N = \{\bP_t^N\}_{t \in \Tb} \in \overline{\Pf}^N := \Pf(\mu_0^N) \otimes \Pf(\mu_1^N) \otimes \cdots$, we define $\Pb^{\overline{\bP}^N}$ as the product measure of all $\Pb^{\bP_t^N}$ such that each agent  transitions independently according to
\begin{align}
	\label{eq:X_N}
	\Pb^{\overline{\bP}^N} \left(X_{t+1}^{N, j} = y \big| X_t^{N,j} = x\right) = \int_U \, p_t(x, u, \mu_t^N)[y] \, \pi_t^{N,j} (\dd u)
\end{align}  
for all $j \in \Nc$ and $x, y \in \Dc$.

Similar to the MFG, we are interested in identifying stationary equilibria to the $N$-agent game. To that end, we fix Agent $i$ as the representative agent (for some $i \in \Nc$) who applies a control $\pi^{N, i}$ and assume that all other $N-1$ agents apply a common control $\pi^{N, -i}$. We then define Agent $i$'s robust utility function in the $N$-agent game by 
\begin{align}
	\label{eq:V_N}
	V^{N, \pi^{N, i}, \pi^{N, -i}} (x, v) &= \inf_{\overline{\bP}^{N} \in \overline{\Pf}^{N}} \; \Eb^{\overline{\bP}^{N}}_{x, v} \left[ \sum_{t=0}^{\infty} \rho^t \cdot r \left(X_t^{N,i},  \pi_t^{N, i},  \mu_t^N, X_{t+1}^{N,i} \right)\right], 
\end{align}
where $X^{N,i}$ is governed by \eqref{eq:X_N}.

\begin{remark} 
	The above $N$-agent game is similar to the MFG introduced in Subsection \ref{sub:MFG}, but there is a key difference between the two models, concerning the evolution of the population distributions. In the MFG, given a transition kernel $\bP_t = \{p_t(x, u, \mu_t)\}_{(x, u) \in \Dc \times U}$, the population distribution $\mu_t$, and the common control $\pi_t$ at time $t$, the distribution in the next period, $\mu_{t+1}$, is fully determined by \eqref{eq:mu}, and thus, once $\{\bP_t\}_{t \in \Tb}$ and $\pi = \{\pi_t\}_{t \in \Tb}$ are fixed, the population distribution evolves \emph{deterministically} according to \eqref{eq:mu}. However, the same is \emph{not} true in the $N$-agent game. Note that in a parallel setting, for fixed $\bP_t^N$ and $\pi_t^{N,j}$ at time $t$, although agent j's state $X^{N,j}_{t+1}$ is governed by \eqref{eq:X_N}, the empirical distribution $\mu_{t+1}^N$ remains \emph{random} at time $t$, due to the independent transitions of finitely many agents. This randomness vanishes in the 
	limit $N \to \infty$ by concentration of measure (Lemma~\ref{lem:distribution-convergence}), establishing the connection between the two formulations.
\end{remark}

Following the same approach as in Definition~\ref{def:ad_MFG}, 
we define a stationary equilibrium to the $N$-agent game as follows. 

\begin{definition}
	\label{def:n-agent-equilibria}
	Suppose that Agent $i$ follows an admissible strategy $\pi^{N, i} \in \Pi_N^{\Tb}$, while all other agents apply a common admissible, time-homogeneous strategy $\pi^{N, *} \in \Pi_N$ (or $\pi^{N,\varepsilon} \in \Pi_N$). $\pi^{N, *}$ is called a stationary equilibrium to the $N$-agent game if it satisfies the following full optimality condition: 
	\begin{align}
		\label{eq:pi_op_N}
		V^{N, \pi^{N, *}, \pi^{N, *}}(x, v) = \sup_{\pi^{N,i} \in \Pi_N^{\Tb}} \, V^{N, \pi^{N, i}, \pi^{N, *}}(x, v), \quad \text{for all } (x, v) \in \Delta_N,
	\end{align}
	where $V^{N, \pi^{N,i}, \pi^{N, *}}$ is defined by \eqref{eq:V_N}.	
	For $\varepsilon\geq0$, we call $\pi^{N,\varepsilon}$ a stationary $\varepsilon$-equilibrium if \eqref{eq:pi_op_N} is replaced by 
	\begin{align}
		\label{eq:pi_eps_N}
		V^{N, \pi^{N,\varepsilon}, \pi^{N,\varepsilon}}(x, v) \ge \sup_{\pi^{N,i} \in \Pi_N^{\Tb}} \, V^{N, \pi^{N,i}, \pi^{N,\varepsilon}}(x, v) - \varepsilon, \quad \text{for all } (x, v) \in \Delta_N.
	\end{align} 
\end{definition}

In Definition~\ref{def:n-agent-equilibria}, Agent $i$'s strategy $\pi^{N, i}$ may deviate from the common strategy $\pi^{N, *}$ (or $\pi^{N,\varepsilon}$) at \emph{all} time points $t \in \Tb$, and for that reason, we call \eqref{eq:pi_op_N} (or \eqref{eq:pi_eps_N}) the \emph{full} optimality condition. Using the same reasoning as in Proposition~\ref{prop:one_shot_eps}, we can characterize equilibria via the simpler \emph{one-shot} optimality condition, under which Agent $i$ adopts $\pi^{N,i} \otimes_1 \pi^{N,*}$ (or $\pi^{N,i} \otimes_1 \pi^{N,\varepsilon}$). We state the result below but omit the proof, as it is analogous to that of Proposition~\ref{prop:one_shot_eps}. 

\begin{proposition}
	\label{prop:n-agent-one-shot}
	Let $\pi^{N, *} \in \Pi_N$ be an admissible, time-homogeneous strategy. Then, the full optimality condition in \eqref{eq:pi_op_N} for a stationary equilibrium is equivalent to the one-shot optimality condition:
	\begin{align*}
		V^{N, \pi^{N, *}, \pi^{N, *}}(x, v) = \sup_{\pi^{N,i} \in \Pi_N} \, V^{N, \pi^{N,i} \otimes_1 \pi^{N,*}, \pi^{N, *}}(x, v), \quad \text{for all } (x, v) \in \Delta_N.
	\end{align*} 
	For approximate equilibria, the following results hold:
	\begin{enumerate}
		\item[(i)] If $\pi^{N, \varepsilon}$ is a stationary $\varepsilon$-equilibrium, then it satisfies the $\varepsilon$-one-shot optimality condition:
		\begin{align*}
			V^{N, \pi^{N, \varepsilon}, \pi^{N, \varepsilon}}(x, v) \ge \sup_{\pi^{N,i} \in \Pi_N} \, V^{N, \pi^{N,i} \otimes_1 \pi^{N,\varepsilon}, \pi^{N, \varepsilon}}(x, v) - \varepsilon, \quad \text{for all } (x, v) \in \Delta_N.
		\end{align*}
		\item[(ii)] Conversely, if $\pi^{N, \varepsilon}$ satisfies the $\varepsilon$-one-shot optimality condition above, then it is a stationary $\frac{2\varepsilon}{1-\rho}$-equilibrium.
	\end{enumerate}
\end{proposition}

\subsection{Main Results}
\label{sub:main}

The previous subsections introduced the MFG and the corresponding $N$-agent game, both under model uncertainty, and studied them separately. In this subsection, we study the convergence between the two games when the number of agents $N$ tends to infinity, and establish the existence of stationary equilibria to the $N$-agent game.

Recall from Section \ref{sec:inf_MFG} that Markov (feedback) controls are functionals from the space $\Pi$ in \eqref{eq:Pi}, and every $\phi \in \Pi$ is a Borel-measurable mapping from the state-distribution space $\Delta$ to $\Pc(U)$, which is the space of distributions on the action space $U$. In the subsequent analysis, we focus on strategies that are continuous in the distribution argument:
\begin{align}
	\label{eq:Pi_c}
	\Pi_c := \{ \phi \in \Pi : v \mapsto \phi(x, v) \text{ is continuous for every } x \in \Dc \}.
\end{align}
The continuity requirement in $\Pi_c$ is natural in our setting: since agents observe the population distribution and respond optimally, continuous dependence on the distribution reflects the intuition that small changes in the population should not induce drastic strategy changes. 

The results below concern the convergence between the MFG and the corresponding $N$-agent game. In the first theorem, we show that an MFG equilibrium is an $\varepsilon$-equilibrium to the $N$-agent game for sufficiently large $N$.

\begin{theorem}
	\label{thm:mf-limit}
	Let Assumptions~\ref{ass:r} and~\ref{ass:uncertainty} hold, and suppose that $\pi^* \in \Pi_c$ is a stationary equilibrium to the MFG as in Definition~\ref{def:eq}. Then, for every $\varepsilon > 0$, there exists an $N_0$ such that for all $N \ge N_0$, $\pi^*$ is a stationary $\varepsilon$-equilibrium to the corresponding $N$-agent game as in Definition~\ref{def:n-agent-equilibria}. 
\end{theorem}

The next theorem establishes the converse of Theorem~\ref{thm:mf-limit}, in the sense that any limit of $N$-agent equilibria
$\{\pi^{N,*}\}_{N\geq 2}$ is an MFG equilibrium. 

\begin{theorem}
	\label{thm:convergence-mf}
	Let Assumptions~\ref{ass:r} and~\ref{ass:uncertainty} hold. Suppose that for every $N \geq 2$, $\pi^{N,*} \in \Pi_N$ is a stationary equilibrium  for the $N$-agent game as in Definition~\ref{def:n-agent-equilibria}. If $\pi^{\infty} \in \Pi_c$ satisfies 
	\begin{align*}
		\lim_{N \to \infty} \sup_{(x, v) \in \Delta_N} d_{W_1}(\pi^{N, *}(x, v), \pi^\infty(x, v)) = 0,
	\end{align*}
	then $\pi^{\infty}$ is a stationary equilibrium to the MFG as in Definition~\ref{def:eq}. 
\end{theorem}

Together, Theorems~\ref{thm:mf-limit} and~\ref{thm:convergence-mf} establish a two-way justification of the mean-field framework: it emerges as the limit of finite-population equilibria as $N \to \infty$, and in turn provides asymptotically exact solutions for large but finite games. 

Note that Theorem~\ref{thm:convergence-mf} assumes the existence of a limit point $\pi^\infty \in \Pi_c$ for the sequence $\{\pi^{N,*}\}$. The following corollary removes this assumption by imposing uniform equicontinuity on $\{\pi^{N,*}\}_{N \geq 2}$, which, via the Arzel\`a-Ascoli theorem, guarantees the existence of a convergent subsequence.

\begin{corollary}
	\label{cor:subsequence-convergence}
Let Assumptions~\ref{ass:r} and~\ref{ass:uncertainty} hold. Suppose that for $N \in \Nb $, $\pi^{N,*} \in \Pi_N$ is a stationary equilibrium  for the $N$-agent game as in Definition~\ref{def:n-agent-equilibria}. If, for every $x \in \Dc$, the family $\{\pi^{N,*}(x, \cdot)\}_{N \geq 2}$ is uniformly equicontinuous as mappings from $(\Pc(\Dc), d_{W_1})$ to $(\mathcal{P}(U), d_{W_1})$, then there exists a subsequence $\{\pi^{N_k, *}\}_{k \in \mathbb{N}}$ that converges to some $\pi^\infty \in \Pi_c$, and this $\pi^\infty$ is a stationary equilibrium to the MFG as in Definition~\ref{def:eq}. 
\end{corollary}

\begin{proof}
    Since $U$ is compact, the space $(\mathcal{P}(U), d_{W_1})$ is compact, so the family $\{\pi^{N,*}(x, \cdot)\}_{N \geq 2}$ is uniformly bounded. Combined with the assumed equicontinuity, the Arzel\`a-Ascoli theorem yields a uniformly convergent subsequence with limit $\pi^\infty \in \Pi_c$. The result then follows from Theorem~\ref{thm:convergence-mf}.
\end{proof}

So far, all of our results rely on the assumption that either MFG or $N$-agent game equilibria exist. This motivates us to establish an existence result for the $N$-agent game.

\begin{theorem}
	\label{thm:n-agent-existence}
	Let Assumptions~\ref{ass:r} and~\ref{ass:uncertainty} hold. For every finite $N \geq 2$, there exists a stationary equilibrium $\pi^{N,*} \in \Pi_N$ to the $N$-agent game.
\end{theorem}

Unlike the finite‑player game, the general existence of mean‑field equilibria remains an open question within our framework. The main obstacles are twofold. First, the natural strategy space $\Pi_c$ defined in \eqref{eq:Pi_c} lacks compactness under any reasonable topology. Second, even if compactness could be restored by restricting to a suitable subset, the best‑response correspondence may still not map back into $\Pi_c$: while the pointwise optimization problem $\arg\max_{\sigma \in \mathcal{P}(U)} V^{\sigma \otimes_1 \pi, \pi}(x,v)$ admits a solution for each $(x,v)$, Berge’s maximum theorem only guarantees upper hemicontinuity of the arg‑max correspondence—this does not ensure the existence of a continuous selection $(x,v) \mapsto \pi^*(x,v)$. In contrast, the $N$‑agent game avoids both issues: its state space $\Delta_N$ is finite and discrete, so pointwise maximization directly yields an admissible strategy in $\Pi_N$ without any continuity requirement.

Nevertheless, Theorem~\ref{thm:convergence-mf} offers an indirect route to existence: any limit point of $N$-agent equilibria (whose existence is guaranteed by Theorem~\ref{thm:n-agent-existence}) must be a mean‑field equilibrium. Indeed, in the next section we construct concrete examples where MFG equilibria according to Definition~\ref{def:eq} do exist.

\section{Examples}
\label{sec:exm}

\subsection{A Solvable MFG}
\label{sub:sol}
For the general Markov MFG under model uncertainty in Subsection \ref{sub:MFG}, an existence result for stationary equilibria is not available. To demonstrate that the proposed MFG model is reasonable, we construct a specific MFG within our framework and obtain its closed-form equilibrium. 

\begin{example}
	\label{exm:MFG}
	We set up the MFG model in Subsection \ref{sub:MFG} with the following specifications. It is straightforward to verify that Assumption~\ref{ass:r} is satisfied; Assumption~\ref{ass:uncertainty} follows from the ball structure of $\Ff$ via Proposition~\ref{prop:mapping}.
	\begin{itemize}
		\item the state space is $\Dc = \{1, 2\}$ (that is, the Markov chain has two states), and the space of the population distribution is $\Pc(\Dc) = \{ v \in \Rb_+^2 : v[1] + v[2] = 1 \}$; 
		\item the action space is $U = [0,1]$, and relaxed controls are distributions on $U$, taking values in the set $\Pc(U)$; 
		\item the reward function $r$ is state-independent, given by 
		\begin{align*}
			r(x, u, v, y) \equiv r(u, v)= v[2] \cdot u - \frac{u^2}{2}, \quad v \in \Pc(\Dc), \, u \in U,
		\end{align*}
		implying that agents prefer a higher concentration in state 2 (larger $v[2]$) and that applying a control incurs a quadratic cost.
	\end{itemize}

Before we describe the uncertainty set in detail, recall that $\Ff$ maps $(x, u, v) \in \Dc \times U \times \Pc(\Dc)$ into a set of possible transition probabilities $p(x, u, v)$ for an agent who is in state $x \in \Dc$ and whose control value is $u \in U$, when the population distribution is $v \in \Pc(\Dc)$. We assume that $\Ff$ admits a ``ball structure'' with a state-dependent ``radius'' (see Appendix~\ref{app:ass} for technical details). The ``center'' is given by 
\begin{align*}
	\mathtt{c}(x=1, u, v) = (1-u, u) \in \Pc(\Dc) \quad \text{and} \quad \mathtt{c}(x=2, u, v) = (0, 1) \in \Pc(\Dc), 
\end{align*}
which serves as the ``baseline'' transition probability given the triplet $(x, u, v)$; and  the ``radius'' is given by 
\begin{align*}
	\mathtt{R}(x=1, u, v) = \epsilon > 0 \quad \text{and} \quad \mathtt{R}(x=2, u, v) = 0, 
\end{align*}
which bounds the distance from any $p(x, u, v)\in \Ff(x, u, v)$ to the center $\mathtt{c}(x, u, v)$.
Finally, we consider the following uncertainty set $\Ff$: 
\begin{align*}
	\Ff(x, u, v) = \{ p(x, u, v) \in \Pc(\Dc) : |p(1, u, v)[2] - u| \le \epsilon \text{ and } p(2, u, v)[2] =  1 \},
\end{align*}
where $p(x, u, v)[y]$ denotes the transition probability from state $x$ to state $y$, where $x, y \in \{1, 2\}$. The above $\Ff$, along with the given $\mathtt{c}$ and $\mathtt{R}$, implies that the transition probability from state 1 to state 2 is within $\epsilon$ of the control value $u \in [0,1]$, and that state 2 is an absorbing state with complete certainty, regardless of the control $u$ and distribution $v$.

Under the parameter constraint
    \begin{align}
        \label{eq:param_constraint}
        \rho (2 + \epsilon) = 1,
    \end{align}
    we can derive a closed-form equilibrium, as stated below.
\end{example}

\begin{proposition}
    \label{prop:example_solution}
    Under the setup of Example~\ref{exm:MFG} with the parameter constraint \eqref{eq:param_constraint}, the MFG admits a unique stationary equilibrium characterized as follows:
    \begin{enumerate}
        \item[(i)] The equilibrium strategy is state-independent: $\pi^*(x, v) \equiv \pi^*(v) = \delta_{v[2]}$; i.e., each agent applies control $u^* = v[2]$.
        \item[(ii)] The worst-case transition kernel is given by $\hat{p}(1, u, v)[2] = \max\{0, u - \epsilon\}$ and $\hat{p}(2, u, v)[2] = 1$.
        \item[(iii)] The equilibrium value function is also state-independent:
        \begin{align*}
            V^*(x, v) \equiv V^*(v) = \begin{cases}
                \frac{1}{2(1-\rho)}(v[2])^2, & \text{if } v[2] \le \epsilon, \\[6pt]
                -\frac{\epsilon}{2(1-\rho)} + \frac{1}{2\rho} v[2], & \text{if } v[2] > \epsilon.
            \end{cases}
        \end{align*}
    \end{enumerate}
\end{proposition}

\begin{proof}
By Proposition \ref{prop:one_shot_eps} and the subsequent remark, a stationary equilibrium $\pi^* \in \Pi$ must satisfy the one-shot optimality condition in \eqref{eq:pi_op_one}. Since the reward function $r$ is state-independent, we conjecture that the equilibrium value function is also state-independent, i.e., $V^{\pi^*,\pi^*}(x, v) \equiv V^{\pi^*,\pi^*}(v)$ for $x = 1, 2$ and all $v \in \Pc(\Dc)$. Under this ansatz, the one-shot optimality condition \eqref{eq:pi_op_one} requires
\begin{align}
	\label{eq:one_shot_exm}
	\pi^* \in \arg\max_{\pi \in \Pi} \; \inf_{\bP_0 \in \Pf(v)} \; \Eb^{\bP_0}_{x, v} \left[ r(\pi, v) + \rho \, V^{\pi^*,\pi^*}(\mu_1)\right].
\end{align}
Observe that for every $\bP_0 = \{p(x, u, v)\}_{(x,u) \in \Dc \times U} \in \Pf(v)$,
\begin{align*}
	\Eb^{\bP_0}_{x, v} \left[ r(\pi, v) + \rho \, V^{\pi^*,\pi^*}(\mu_1)\right] = \int_0^1 r(u, v) \, \pi(\dd u) + \rho \, V^{\pi^*,\pi^*}(\mu_1),
\end{align*}
where the first integral is \emph{independent} of $\bP_0$, while $V^{\pi^*,\pi^*}(\mu_1)$ depends on $\bP_0$ through the population distribution $\mu_1$ via \eqref{eq:mu1}. Consequently, the $\sup$ and $\inf$ in \eqref{eq:one_shot_exm} can be separated:
\begin{align*}
	\pi^* \in \arg\max_{\pi \in \Pi} \left\{ \int_0^1 r(u, v) \, \pi(\dd u) + \rho \inf_{\bP_0 \in \Pf(v)} V^{\pi^*,\pi^*}(\mu_1) \right\}.
\end{align*}
Because  the second term does not depend on $\pi$, the equilibrium strategy $\pi^*$ maximizes only the reward integral:
\begin{align*}
	\pi^*(v) \in \arg\max_{\pi \in \Pc(U)} \int_0^1 \left( v[2] \cdot u - \frac{u^2}{2} \right) \pi(\dd u).
\end{align*}
For fixed $v$, the integrand $v[2] \cdot u - u^2/2$ is strictly concave in $u$ and attains its maximum at $u^* = v[2] \in [0,1]$. Therefore, $\pi^*(v) = \delta_{v[2]}$, i.e., the equilibrium strategy is a Dirac measure concentrated at $u^* = v[2]$. Note that $\pi^*$ is independent of the current state $x$ and only depends on the current population distribution $v$. As $\pi^*$ is a point mass, we will write $\pi^* = u^*(v) = v[2]$ as the equilibrium strategy in the rest of this example.

Having established that the equilibrium strategy is a point mass at $u^* = v[2]$, we can now analyze the evolution of the population distribution and the $\inf$ problem. Let the initial distribution be $\mu_0 = v = (v[1], v[2]) \in \Pc(\Dc)$. Because all agents in state 1 apply control $u^* = v[2]$, and state 2 is an absorbing state, the distribution $\mu_1$ in the next period is determined by the transition probability $p_{1 \to 2}(v) := p(1, v[2], v)[2]$:
\begin{align}
	\label{eq:mu1_exm}
	\mu_1[1] = v[1] \cdot (1 - p_{1 \to 2}(v)) \quad \text{and} \quad \mu_1[2] = v[1] \cdot p_{1 \to 2}(v) + v[2].
\end{align}

For convenience, denote $V^* (v) := V^{\pi^*,\pi^*}(v)$ the equilibrium value function for all $v \in \Pc(\Dc)$. By combining the above analysis, $V^*$ satisfies the following Bellman equation:
\begin{align}
	\label{eq:value}
	V^*(v) = \frac{1}{2} \big(v[2]\big)^2 + \rho \,  \inf_{\bP_0 \in \Pf(v)} \, V^*(\mu_1),
\end{align}
where $\mu_1$ is given by \eqref{eq:mu1_exm}.

We now analyze the $\inf$ problem in \eqref{eq:value}. Under strategy $u^*(v) = v[2]$, the transition probability $p_{1 \to 2}(v)$ takes values in the closed interval $\mathbb{I}^*_p(v) := [ \max\{0,  v[2] - \epsilon\}, \, \min\{v[2] + \epsilon, 1\} ]$, due to the given uncertainty ``ball'' structure of $\Ff(1, u^*, v)$. Because  the equilibrium strategy $\pi^* = \delta_{v[2]}$ is a point mass at $u^* = v[2]$, the population distribution $\mu_1$ depends on $\bP_0$ only through the value $p(1, u^*, v)[2] = p_{1 \to 2}(v)$. Consequently, the optimization $\inf_{\bP_0 \in \Pf(v)} V^*(\mu_1)$ reduces to $\inf_{p_{1 \to 2}(v) \in \mathbb{I}^*_p(v)} V^*(\mu_1)$.  Using the particular form of the reward function $r$, we expect the value function $V^*$ to be increasing in $v[2]$; under this conjecture, the minimizer is obtained as
\begin{align*}
	p^*_{1 \to 2}(v)  = \max\{0,  v[2] - \epsilon\}.
\end{align*}

To ensure mathematical rigor within our MFG framework, we must verify that this worst-case point $p^*_{1 \to 2}(v)$ can be extended to a full transition kernel $\bP_0^* \in  \Pf(v)$. That is, we need to construct a kernel defined on the entire control space $U$ that matches $p^*_{1 \to 2}(v)$ at $u^*$ while satisfying the global Lipschitz continuity and ball-structure constraints. Inspired by the form of $p^*_{1 \to 2}(v)$, we propose the following candidate kernel: for a given $\mu_0 = v \in \Pc(\Dc)$, 
\begin{align*}
	\hat{p}(1, u, v)[2] = \max\{0, u - \epsilon \} \quad \text{and} \quad 
	\hat{p}(2, u, v)[2]=1,  \quad \text{for all } u \in U.
\end{align*}
Note that when $x=2$, the assumptions on the ball structure dictate that for every $p(2, u, v) \in \Ff(2, u, v)$, it holds that $p(2, u, v)[2]=1$ for all $u \in U$, explaining the second condition above on $\hat{p}$.  We now verify that such a proposal $\hat{p}$ is ``admissible'': 
(1) When $u = u^* = v[2]$, $\hat{p}(1, u^*, v)[2] = \max\{0, v[2] - \epsilon\}   = p^*_{1 \to 2}(v)$; (2) for every $u \in U$, we see that $\hat{p}(x, u, v) \in \Ff(x, u, v)$ for both $x=1, 2$; (3) the mapping $u \mapsto \hat{p}(1, u, v)[2]$ is $1$-Lipschitz continuous (so is $u \mapsto \hat{p}(2, u, v)[2]$), satisfying the regularity requirement of $\Pf$ in  Assumption~\ref{ass:uncertainty}. 
Therefore, the proposed candidate is a valid worst-case transition kernel $\bP^*_0$. 

In the last step, we plug in the obtained optimal $u^*$ and $\bP^*_0$ into \eqref{eq:value} and solve for the equilibrium value function $V^*$, and we verify that it satisfies all assumed properties. We consider two cases based on whether $v[2] \leq \epsilon$ or $v[2] > \epsilon$.
\begin{enumerate}
	\item Case 1: $v[2] \le \epsilon$.  In this case, $p^*_{1 \to 2}(v)  = 0$, and the population distribution remains unchanged, with $\mu_1 = \mu_0 = v$. The Bellman equation in \eqref{eq:value} reduces to 
	\begin{align*}
		V^*(v) = \frac{1}{2} \big(v[2]\big)^2 + \rho \, V^*(v),
	\end{align*}
	implying that 
	\begin{align*}
		V^*(v) = \frac{1}{2(1 - \rho)} \big(v[2]\big)^2,  \quad \text{for all } v \in \Pc(\Dc) \text{ with } v[2] \le \epsilon.
	\end{align*}
	
	\item Case 2: $v[2] > \epsilon$.   In this case, $p^*_{1 \to 2}(v)  = v[2] - \epsilon$, and the proportion of the population in state 2 transits from $\mu_0[2] = v[2]$ to  
	\begin{align*}
		\mu_1[2] = \big(1 - v[2]\big) \, \big(v[2] - \epsilon\big) + v[2] = -  \big(v[2]\big)^2 + (2 + \epsilon) \cdot v[2] - \epsilon. 
	\end{align*}
	To solve \eqref{eq:value}, we consider an ansatz for $V^*$ in the form of 
	\begin{align*}
		V^*(v) = A + B \cdot v[2], 
	\end{align*}
	where $A$ and $B$ are two constants, with $B > 0$. Under the parameter constraint \eqref{eq:param_constraint}, the above ansatz solves \eqref{eq:value} with 
	\begin{align*}
		A = - \frac{\epsilon}{2 ( 1 - \rho)} \quad \text{and} \quad B = \frac{1}{2 \rho} > 0. 
	\end{align*}
\end{enumerate}

It is straightforward to verify that the above $V^{*}$ is continuous at $v[2] = \epsilon$ and increasing in $v[2]$, confirming the monotonicity assumption used in our derivation.
\end{proof}

To summarize, we construct a solvable MFG in this example and obtain the stationary equilibrium strategy (see Definition~\ref{def:eq}) and the worst-case transition kernel, along with the corresponding equilibrium value function, all in closed form. This example provides concrete support for the MFG framework
with model uncertainty proposed in Subsection \ref{sub:MFG}; in particular, it shows that the equilibrium concept in  Definition~\ref{def:eq} can lead to a meaningful solution. 

Finally, we note that the corresponding $N$-agent game under the same model specifications does not admit a closed-form solution. This is because the value function, $V^{N, \pi^{N,*}, \pi^{N,*}}(x, v)$, depends on both the current state $x$ and distribution $v$; in comparison, the value function in the MFG $V^*(v)$ only depends on $v$. This challenge in the $N$-agent game further demonstrates the usefulness of studying the MFG, because Theorem~\ref{thm:mf-limit} guarantees that the MFG equilibrium $\pi^*$ is an $\varepsilon$-equilibrium to the $N$-agent game when the number of agents exceeds a bound ($N \ge N_0 := N_0(\varepsilon)$). Moreover, we know from  Theorem~\ref{thm:n-agent-existence} that an equilibrium $\pi^{N,*}$ to the $N$-agent game always exists, and even though finding it can be extremely difficult, Theorem~\ref{thm:convergence-mf} ensures that any convergent sequence of such equilibria, $\{ \pi^{N,*} \}_{N\geq 2}$, must converge to the MFG equilibrium $\pi^*$, which we have obtained.

\subsection{A Comparison Example}
\label{sub:com}
Both \citet{langner2024markov} and this paper study Markov MFGs under model uncertainty, but with one fundamental difference: an equilibrium is a closed-loop (feedback) control in our setup (see Definitions~\ref{def:ad_MFG} and \ref{def:eq}), but is an open-loop control in theirs.  In their framework, transition probabilities are treated as part of an equilibrium, and thus they are determined endogenously; for that reason, we refer to their equilibrium as the \emph{endogenous equilibrium}. In contrast, our framework treats transition probabilities as \emph{exogenously} given, but uncertain, belonging to a constraint set that captures imperfect knowledge about the system dynamics. In this subsection, we conduct a detailed comparison between these two setups, examining how this key difference affects the time-consistency 
properties and the adaptability of equilibrium strategies when the realized path deviates from the forecast one.

Let us start with a recap on the key setup of \citet{langner2024markov}, and we often add an overhead \texttt{tilde} $\widetilde{\ }$ in notation when reviewing their setup. In a finite horizon with $T$ periods ($\Tb_T = \{0,1, \cdots, T\}$), let $\widetilde{\mu} = \{\widetilde{\mu}_t\}_{t \in \Tb_T}$ denote a full path of the population distribution and $\widetilde{\pi}= \{\widetilde{\pi}_t\}_{t \in \Tb_T}$ an admissible strategy. For a given pair $(\widetilde{\mu}, \widetilde{\pi})$, $\Pc(\widetilde{\mu}, \widetilde{\pi})$ denotes the set of all transition probabilities  $\widetilde{p}$ that are consistent with the distribution path $\widetilde{\mu}$ and strategy $\widetilde{\pi}$. The  worst-case objective function is defined over a \emph{fixed} distribution path by 
\begin{align*}
	J \left(\widetilde{\pi}, \widetilde{\mu} \right) = \inf_{\widetilde{p} \in \Pc(\widetilde{\mu}, \widetilde{\pi})} \mathbb{E}^{\widetilde{p}}\left[\sum_{t=0}^{T-1} \rho^t r \big( \widetilde{X}_t, \widetilde{\pi}_t,  \widetilde{\mu}_t, \widetilde{X}_{t+1} \big) + \rho^T g(\widetilde{X}_T, \widetilde{\mu}_T)\right],
\end{align*}
where $\widetilde{X}_t$ is the state process with transition governed by $\widetilde{p}$,  the expectation is taken under the probability measure induced by $\widetilde{p}$, and $r$ and $g$ are the running and terminal reward functions, respectively. 
The agent's goal is to seek an optimal strategy over the admissible set $\widetilde{\Pi}$ to maximize the objective function $J$, 
\begin{align*}
	\widetilde{V}(\widetilde{\mu}) = \sup_{\widetilde{\pi}  \in \widetilde{\Pi} } \,  J(\widetilde{\mu},  \widetilde{\pi}).
\end{align*}
An endogenous equilibrium is a triplet $(\widetilde{p}^*,\widetilde{\pi}^*,\widetilde{\mu}^*)$ such that (i) $\widetilde{p}^*$ achieves the infimum in $J(\widetilde{\mu}^*,\widetilde{\pi}^*)$, (ii)  $\widetilde{\pi}^*$ achieves the supremum in $\widetilde{V}(\widetilde{\mu}^*)$, and (iii) $\widetilde{\mu}^*$ is generated by $(\widetilde{p}^*,\widetilde{\pi}^*)$. 

As seen from the above description, the endogenous framework implicitly assumes that agents \emph{precommit} to an equilibrium strategy $\widetilde{\pi}^*$ over the \emph{entire path} of the distribution $\widetilde{\mu}^*$. This formulation raises two conceptual issues when interpreted from the perspective of robust individual optimization.

First, the endogenous equilibrium exhibits a \emph{precommitment} structure. Under model uncertainty, the realized population distributions $\widetilde{\mu}$ may differ from the equilibrium path $\widetilde{\mu}^*$ at later time periods due to uncertain transition probabilities. Since the equilibrium strategy $\widetilde{\pi}^*$ is computed with respect to the fixed path $\widetilde{\mu}^*$, it does not prescribe how agents should respond when the realized distribution deviates from this path. In this sense, the endogenous framework provides a \emph{precommitted} rather than an \emph{adaptive} strategy.

Second, the endogenous framework defines a value function $\widetilde{V}(\widetilde{\mu}_{0:T})$ that depends on the \emph{entire population path} and optimizes strategies under this fixed path. From a game-theoretic perspective, this differs from the standard approach where each agent's utility depends on the \emph{realized} outcomes rather than a presumed distributional path. Since an individual agent controls only its own state and action, one may question whether the resulting equilibrium properly reflects individual incentives when the population path is not perfectly predictable.

Our framework adopts a different approach. We consider a representative agent who applies a relaxed control $\pi^i$, where at every time $t$, $\pi^i_t := \pi^i_t(X_t^i, \mu_t)$ is a Markov (feedback) strategy depending on the state-distribution pair. We define an equilibrium $\pi^*$ as the one that maximizes the representative agent's robust utility 
$V^{\pi^i, \pi^*}$, assuming that all other agents adopt $\pi^*$. Under this definition, the equilibrium strategy $\pi_t^*$ adapts to the current \emph{realized} state-distribution pair $(X_t, \mu_t)$, providing an adaptive response mechanism to distributional uncertainty.

Given these conceptual differences between the two frameworks, a natural question arises: could the endogenous equilibrium strategy $\widetilde{\pi}^*$ still constitute an equilibrium under our robust framework?
We construct an explicit example below to show that the endogenous equilibrium $\widetilde{\pi}^*$ is \emph{not} an equilibrium under our robust framework, in the sense that an agent can strictly improve its utility by deviating from $\widetilde{\pi}^*$ when all other agents adopt $\widetilde{\pi}^*$.

\begin{example}
	\label{exm:com}
	Consider a one-period model ($T=1$) with the same specifications as in Example \ref{exm:MFG}. In addition, the terminal reward function is given by $g(x, v) = \frac{1}{2} (v[2])^2 \cdot \mathbf{1}_{x=1}$ (which is only payable if the Markov chain is in state 1 at time 1); the initial distribution $\mu_0$ is given as $v = (v[1], v[2])$, and $v[2]$ satisfies $v[2] \in [\epsilon, 1 - \epsilon]$. 
\end{example}

\begin{proposition}
Under the setup of Example~\ref{exm:com}, the endogenous equilibrium strategy $\widetilde{u}^*$ from \citet{langner2024markov} is \emph{not} a Nash equilibrium under our robust framework. Specifically, there exists a deviation $u^i \neq \widetilde{u}^*$ such that
    \[
    V^{u^i, \widetilde{u}^*}(x=1, v) > V^{\widetilde{u}^*,\widetilde{u}^*}(x=1, v).
    \]
\end{proposition}
\begin{proof}
Because the model has only one period and state 2 is an absorbing state, the Markov system is fully determined by the transition probability from state 1 at time 0 to state 2 at time 1 ($\Pb(\widetilde{X}_0 = 1 \to \widetilde{X}_1 = 2)$), which we denote by $p$ for simplicity.  Note that for a realized control value $u \in U$, $p := p(u)$ takes values in $\mathbb{I}_p(u) := [\max\{0, u - \epsilon\}, \min\{1, u + \epsilon\}]$ by the ball structure assumption on the uncertainty set (see the analysis of Example~\ref{exm:MFG}). 

We first derive the endogenous equilibrium  $(\widetilde{p}^*,\widetilde{\pi}^*,\widetilde{\mu}^*_1)$ under the endogenous framework, where $\widetilde{p}^* = \Pb(\widetilde{X}_0 = 1 \to \widetilde{X}_1 = 2)$, $\widetilde{\pi}^*$ is a relaxed control applied at time 0, and $\widetilde{\mu}_1^*$ is the population distribution at time 1, given $\mu_0 = v$. In the subsequent analysis, we consider agents who are in state 1 at time 0 ($\widetilde{X}_0 = x = 1$); note that agents who are in state 2 at time 0 will remain in state 2 at time 1, and thus will not receive the terminal reward. A key observation is that with the population path $\{\mu_0, \widetilde{\mu}^*_1\}$ fixed, the terminal reward depends on the transition probability $\widetilde{p}(u)$ only via the terminal state $\widetilde{X}_1$. Consequently, we obtain 
\begin{align*}
	\mathbb{E}^{\widetilde{p}} \left[g (\widetilde{X}_1, \widetilde{\mu}_1^*)\right] =  \mathbb{E}^{\widetilde{p}} \left[ \mathbf{1}_{\widetilde{X}_1 = 1}  \right] \cdot \frac{(\widetilde{\mu}_1^*[2])^2}{2} = (1 - \widetilde{p}(u)) \cdot  \frac{(\widetilde{\mu}_1^*[2])^2}{2}.  
\end{align*}

The robust optimization problem under the endogenous framework then reads as  
\begin{align}
	\label{eq:langner_opt}
	\sup_{\widetilde{\pi} \in \Pi} \, \int_0^1 \left[ v[2] \cdot u - \frac{u^2}{2} + \rho \inf_{ \widetilde{p} (u) \in \mathbb{I}_p(u)} (1- \widetilde{p} (u)) \cdot \frac{(\widetilde{\mu}_1^*[2])^2}{2} \right] \widetilde{\pi}(\dd u).
\end{align}
Since $(1 - \widetilde{p}(u))$ is decreasing in $\widetilde{p}(u)$, the infimum is achieved at the largest feasible value: $\widetilde{p}^*(u) = \min\{1, u + \epsilon\}$. We first consider the case $u \le 1 - \epsilon$, where $\widetilde{p}^*(u) = u + \epsilon$, and the integrand in \eqref{eq:langner_opt} becomes
\begin{align*}
	f(u) := v[2] \cdot u - \frac{u^2}{2} + \frac{\rho (\widetilde{\mu}_1^*[2])^2}{2} (1 - u - \epsilon).
\end{align*}
Taking the derivative $f'(u) = v[2] - u - \frac{\rho (\widetilde{\mu}_1^*[2])^2}{2}$ and setting it to zero yields the candidate optimizer
\begin{align}
	\label{eq:u_star_endo}
	\widetilde{u}^* = v[2] - \frac{\rho (\widetilde{\mu}_1^*[2])^2}{2}.
\end{align}
As the integrand is strictly concave in $u$, the supremum in \eqref{eq:langner_opt} is achieved by the point mass $\widetilde{\pi}^* = \delta_{\widetilde{u}^*}$. Applying the consistency condition on the terminal distribution:
\begin{align}
	\label{eq:en_mu}
	\widetilde{\mu}_1^*[2] = v[2] + v[1] \cdot \widetilde{p}^*(\widetilde{u}^*) = v[2] + (1 - v[2]) \cdot (\widetilde{u}^* + \epsilon).
\end{align}
Substituting \eqref{eq:u_star_endo} into \eqref{eq:en_mu}, we obtain a fixed-point equation for $\widetilde{\mu}_1^*[2]$:
\begin{align}
	\label{eq:fixed_point}
	\widetilde{\mu}_1^*[2] = v[2] + (1 - v[2]) \cdot \left( v[2] - \frac{\rho (\widetilde{\mu}_1^*[2])^2}{2} + \epsilon \right) =: \widetilde{\phi}(\widetilde{\mu}_1^*[2]).
\end{align}
As $\widetilde{\phi}'(z) = -\rho (1-v[2]) z < 0$ for $z > 0$, the mapping $\widetilde{\phi}$ is strictly decreasing. To establish the existence of a fixed point, we verify that $\widetilde{\phi}(z_0) > z_0$ for $z_0 := v[2]+(1-v[2]) \epsilon$. This is equivalent to showing $v[2] > \frac{\rho}{2} z_0^2$. Indeed, using $\rho \le 1$, $z_0 \le 1$, and $\epsilon \le v[2]$, we have
\begin{align*}
	\frac{\rho}{2} z_0^2 \le \frac{1}{2} z_0 \le \frac{1}{2} \big( v[2] + (1-v[2]) v[2] \big) = \frac{1}{2} v[2] (2 - v[2]) < v[2].
\end{align*}
Similarly, $\widetilde{\phi}(1) = v[2] + (1-v[2])(v[2] + \epsilon - \frac{\rho}{2}) < 1$ is straightforward. By the intermediate value theorem, there exists a unique fixed point $\widetilde{\mu}_1^*[2] \in (z_0,1)$. From \eqref{eq:u_star_endo}, the corresponding equilibrium strategy satisfies $\widetilde{u}^* \in (0, v[2]]$. As $v[2] \le 1 - \epsilon$ by assumption, we have $\widetilde{u}^* \le 1 - \epsilon$, confirming that the solution lies within the case considered above.

Next, we show that an agent can achieve a strict gain by deviating from $\widetilde{\pi}^*$ when evaluated under our robust framework. Assume that Agent $i$ is in state 1 at time 0 and uses a regular control $u^i \in [0,1]$ (that is, $\pi^i = \delta_{u^i}$) and that all other agents apply the endogenous equilibrium strategy $\widetilde{u}^*$ (that is, $\pi = \delta_{\widetilde{u}^*}$). Different from the endogenous setup, the terminal distribution $\mu_1$ in our framework is \emph{not} fixed a priori, but rather depends on the transition kernel $\bP = \{p(x, u)\}_{(x, u) \in \{1,2\} \times [0,1]}$, where we have suppressed the dependence on the initial distribution $\mu_0 = v$ in notation. Because $p(2, u)[1] = 0$ and $p(2, u)[2] = 1$ for all $u \in [0,1]$ (state 2 is an absorbing state), the kernel $\bP$ is fully determined by $P(\cdot) :=\{p(1, u)[2]\}_{u \in [0,1]}$. The uncertainty set $\Pf(v)$ requires that $P(u) \in \mathbb{I}_p(u)$ for all $u \in U$, and that $P(\cdot)$ is $L$-Lipschitz continuous (with $L = 1$ in this example).

For every $P(\cdot) \in \Pf(v)$, the terminal distribution and Agent $i$'s survival probability are given by
\begin{align*}
	\mu_1[2] = v[2] + (1 - v[2]) \cdot P(\widetilde{u}^*) \quad \text{and} \quad 
	\Eb^{P(\cdot)} \left[ \mathbf{1}_{X^i_1 = 1}  \right] = 1 - P(u^i).
\end{align*}
Therefore, Agent $i$'s utility function under our framework is obtained by
\begin{align}
	\label{eq:Vi_exm}
	V^{u^i, \widetilde{u}^*}(x=1, v) = v[2] \cdot u^i - \frac{(u^i)^2}{2} + \frac{\rho}{2} \, \inf_{P(\cdot) \in \Pf(v)} \, (1 - P(u^i)) \cdot \left( v[2] + (1 - v[2]) \cdot P(\widetilde{u}^*)  \right)^2.
\end{align}
As  both $\pi^i = \delta_{u^i}$ and $\pi = \delta_{\widetilde{u}^*}$ are point masses, the population distribution $\mu_1$ depends on $P(\cdot)$ only through the single value $P(\widetilde{u}^*)$, and Agent $i$'s terminal state depends only on $P(u^i)$. Thus, the infimum in \eqref{eq:Vi_exm} reduces to an optimization over two scalars $(P(u^i), P(\widetilde{u}^*))$, subject to: (i) $P(u^i) \in \mathbb{I}_p(u^i)$, (ii) $P(\widetilde{u}^*) \in \mathbb{I}_p(\widetilde{u}^*)$, and (iii) the Lipschitz constraint $|P(u^i) - P(\widetilde{u}^*)| \le |u^i - \widetilde{u}^*|$.

The worst-case scenario faces competing objectives: (1) choose large $P(u^i)$ to reduce Agent $i$'s survival probability $(1-P(u^i))$; (2) choose small $P(\widetilde{u}^*)$ to reduce the terminal distribution $\mu_1[2]$ and hence the reward value. We analyze the infimum problem in two cases.

\emph{Case 1: $u^i = \widetilde{u}^*$.} When Agent $i$ uses the same control as the population, the infimum involves only a single variable $P := P(\widetilde{u}^*) \in \mathbb{I}_p(\widetilde{u}^*)$:
\begin{align*}
	\inf_{P \in \mathbb{I}_p(\widetilde{u}^*)} (1 - P) \cdot \left( v[2] + (1 - v[2]) \cdot P \right)^2.
\end{align*}

This is a cubic function in $P$. A straightforward calculation shows that this cubic has no critical points in $[0,1]$, so the minimum is attained at one of the boundary points of $\mathbb{I}_p(\widetilde{u}^*)$.

\emph{Case 2: $u^i \neq \widetilde{u}^*$.} When Agent $i$ uses a different control $u^i$, the infimum in \eqref{eq:Vi_exm} involves two variables $(P(u^i), P(\widetilde{u}^*))$, subject to the constraints: (i) $P(u^i) \in \mathbb{I}_p(u^i)$, (ii) $P(\widetilde{u}^*) \in \mathbb{I}_p(\widetilde{u}^*)$, and (iii) the Lipschitz constraint $|P(u^i) - P(\widetilde{u}^*)| \le |u^i - \widetilde{u}^*|$. The adversary faces competing objectives: maximizing $P(u^i)$ to reduce Agent $i$'s survival probability $(1 - P(u^i))$, while minimizing $P(\widetilde{u}^*)$ to reduce the terminal distribution $\mu_1[2]$. Because these objectives push $P(u^i)$ and $P(\widetilde{u}^*)$ in opposite directions, the Lipschitz constraint must be binding at the optimum. When $u^i > \widetilde{u}^*$, this implies $P(u^i) - P(\widetilde{u}^*) = u^i - \widetilde{u}^*$. Setting $\delta := P(\widetilde{u}^*) - \widetilde{u}^*$, we have $P(u^i) = u^i + \delta$, and the infimum reduces to:
\begin{align*}
	\inf_{\delta} \, (1 - u^i - \delta) \cdot \left( v[2] + (1 - v[2]) \cdot (\widetilde{u}^* + \delta) \right)^2,
\end{align*}
where the feasible range of $\delta$ is determined by the constraints $P(\widetilde{u}^*) = \widetilde{u}^* + \delta \in \mathbb{I}_p(\widetilde{u}^*)$ and $P(u^i) = u^i + \delta \in  \mathbb{I}_p(u^i)$.

To proceed with a numerical demonstration, we set $v[2] = 0.3$, $\epsilon = 0.2$, and $\rho = 0.9$. Solving the fixed-point equation \eqref{eq:fixed_point} yields
\begin{align*}
	\widetilde{\mu}_1^*[2] = 0.5535 \quad \text{and} \quad \widetilde{u}^* = 0.3 - \frac{0.9 \times 0.5535^2}{2} = 0.1621.
\end{align*}

Under our robust framework, if Agent $i$ uses $\widetilde{u}^*$ (Case 1), we have $P \in \mathbb{I}_p(\widetilde{u}^*) = [0, 0.3621]$. Evaluating the cubic $h(P) = (1-P)(v[2] + (1-v[2])P)^2$ at the boundaries gives $h(0) = 0.09$ and $h(0.3621) = 0.1954$, so the infimum is achieved at $P^* = 0$. Therefore,
\begin{align*}
	V^{\widetilde{u}^*,\widetilde{u}^*}(x=1, v) = 0.3 \times 0.1621 - \frac{0.1621^2}{2} + \frac{0.9}{2} \times 0.09 = 0.0760.
\end{align*}

Now consider a deviation to $u^i = 0.26$, which is closer to $v[2] = 0.3$ than $\widetilde{u}^*$. The feasible range for $\delta$ is determined by:
\begin{align*}
	P(\widetilde{u}^*) = \widetilde{u}^* + \delta \ge 0 \;\Rightarrow\; \delta \ge -0.1621, \quad
	P(\widetilde{u}^*) = \widetilde{u}^* + \delta \le 0.3621 \;\Rightarrow\; \delta \le 0.2,
\end{align*}
together with the constraints from $P(u^i) = u^i + \delta \in [0.06, 0.46]$, which are less restrictive. Thus, $\delta \in [-0.1621, 0.2]$.
The objective function $g(\delta) = (1 - u^i - \delta)(v[2] + (1-v[2])(\widetilde{u}^* + \delta))^2$ is a cubic in $\delta$. Taking the derivative and checking the critical points shows that the minimum over $[-0.1621, 0.2]$ is attained at a boundary. We compute
\begin{align*}
	g(-0.1621) &= (1 - 0.26 + 0.1621) \times (0.3 + 0.7 \times 0)^2 = 0.9021 \times 0.09 = 0.0812, \\
	g(0.2) &= (1 - 0.26 - 0.2) \times (0.3 + 0.7 \times 0.3621)^2 = 0.54 \times 0.5535^2 = 0.1654.
\end{align*}
Therefore, the infimum is achieved at $\delta^* = -0.1621$, corresponding to $P^*(\widetilde{u}^*) = 0$ and $P^*(u^i) = 0.0979$. Agent $i$'s utility under $u^i$ is
\begin{align*}
	V^{u^i, \widetilde{u}^*}(x=1, v) &= 0.3 \times 0.26 - \frac{0.26^2}{2} + \frac{0.9}{2} \times 0.0812 = 0.0442 + 0.0365 = 0.0807.
\end{align*}

As $V^{u^i, \widetilde{u}^*} = 0.0807 > V^{\widetilde{u}^*,\widetilde{u}^*} = 0.0760$, Agent $i$ achieves a strictly higher utility by deviating from $\widetilde{u}^*$ to $u^i = 0.26$. The reason is the favorable trade-off between immediate and terminal rewards: the immediate reward increases from $0.035$ to $0.044$ (a gain of $0.009$), while the worst-case terminal reward decreases from $0.041$ to $0.037$ (a loss of $0.004$), yielding a net improvement of approximately $6.2\%$. This demonstrates that the endogenous equilibrium strategy $\widetilde{u}^*$ is \emph{not} a Nash equilibrium when agents evaluate payoffs under our robust framework.
\end{proof}

We now explain why the endogenous equilibrium fails to capture this deviation. The reason lies in its solution procedure: it first fixes a terminal distribution $\widetilde{\mu}_1^*$, then solves the inf-sup problem \eqref{eq:langner_opt} treating $\widetilde{\mu}_1^*$ as given, and finally verifies the consistency condition \eqref{eq:fixed_point}. This circular approach handles the coupling between the transition probabilities and population distribution, reducing the worst-case problem to a one-dimensional optimization over $\widetilde{p}$ alone. Essentially, the endogenous framework solves a fixed-point equation that lacks genuine game-theoretic content, determining a self-consistent triplet $(\widetilde{p}^*,\widetilde{\pi}^*,\widetilde{\mu}_1^*)$ without properly accounting for how individual transitions affect collective outcomes under model uncertainty.

In contrast, our framework treats the population distribution as genuinely stochastic: $\mu_1$ depends on the realized transition kernel $P(\cdot)$ selected by adversarial nature. When Agent $i$ deviates to $u^i \neq \widetilde{u}^*$, the worst-case optimization in \eqref{eq:Vi_exm} involves choosing both $P(u^i)$ and $P(\widetilde{u}^*)$ subject to the Lipschitz constraint, capturing the coupling between individual transitions and the population distribution. The numerical example demonstrates that when all other agents adopt $\widetilde{u}^*$, an individual can strictly improve their robust utility by deviating to $u^i = 0.26$. This shows that the endogenous equilibrium $\widetilde{u}^*$, while self-consistent within its own formulation, does not generally constitute a Nash equilibrium under our robust framework.

\section{Proofs of Main Results}
\label{sec:proof}

\subsection{Preliminaries}
\label{subsec:pre}
In this subsection, we introduce notation and preliminary results that will be used to prove the main results in Subsection \ref{sub:main}. 

For every $\bP = \{p(x, u)\}_{(x, u) \in \Dc \times U} \in \prod_{\Dc \times U} \, \Pc(\Dc)$ and $\phi \in \Pc(U)$, define 
\begin{align*}
	f^{\bP, \phi}(x, y) := \int_U \, p(x, u)[y] \, \phi(\dd u), \quad x, y \in \Dc. 
\end{align*}
Think of $\bP$ as a transition kernel and $\phi$ as a relaxed control, and assume that an agent applies $\phi$ and is currently in state $x$; then, $f^{\bP, \phi}(x, y)$ is the probability that it will transition to state $y$ in the next period. We further extend $\phi$ to be a function, mapping from $\Dc$ to $\Pc(U)$; this extension allows the agent to make decisions $\phi := \phi(x)$ based on its current state $x \in \Dc$. Now given such a function $\phi$ and a transition kernel $\bP$, define 
\begin{align}
	\label{eq:Phi}
	\Phi(\bP, \phi) := \left( f^{\bP, \phi(x)}(x, y)\right)_{x, y \in \Dc} \in [0,1]^{d \times d},
\end{align}
which is a full transition matrix for a controlled Markov chain with the state space $\Dc$. Note that $\Phi(\bP, \phi)[x, y] = f^{\bP, \phi(x)}(x, y)$.

The following lemma establishes the continuity of the mapping $\Phi$ in \eqref{eq:Phi}, which is essential for verifying the joint continuity conditions required when applying Berge's maximum theorem.

\begin{lemma}
	\label{lem:transition-continuity}
	Let $\bP= \{p(x, u)\}_{(x, u) \in \Dc \times U}, \bP^n = \{p^n(x, u)\}_{(x, u) \in \Dc \times U} \in \prod_{\Dc \times U} \, \Pc(\Dc)$ and $\phi, \phi^n: \Dc \to \mathcal{P}(U)$. 
	Assume $p(x,\cdot)[y]: U \to [0,1]$ is $L$-Lipschitz continuous for all $x, y \in \Dc$.
	If $\bP^n \to \bP$ uniformly (that is, $\sup_{x,y,u} |p^n(x,u)[y] - p(x,u)[y]| \to 0$) and $\phi^n \to \phi$ pointwise (that is, $d_{W_1}(\phi^n(x), \phi(x)) \to 0$ for all $x \in \Dc$), then $\Phi(\bP^n, \phi^n) \to \Phi(\bP, \phi)$ entrywise (that is, $\Phi(\bP^n, \phi^n)[x, y] \to \Phi(\bP, \phi)[x, y]$ for all $x, y \in \Dc$).
\end{lemma}

\begin{proof}
	We decompose the difference for each $[x,y]$-entry of the matrices by
	\begin{align*}
		 \left| \Phi(\bP^n, \phi^n) - \Phi(\bP, \phi) \right|[x, y] &= \left|f^{\bP^n, \phi^n(x)}(x, y) -  f^{\bP, \phi(x)}(x, y) \right| \\
		& \le \left| f^{\bP^n, \phi^n(x)}(x, y) -  f^{\bP, \phi^n(x)}(x, y)\right| + \left|f^{\bP, \phi^n(x)}(x, y) -  f^{\bP, \phi(x)}(x, y) \right|. 
	\end{align*}
	
	Regarding the first term, the uniform convergence of $\bP^n \to \bP$ implies that
	\begin{align*}
		\left| f^{\bP^n, \phi^n(x)}(x, y) -  f^{\bP, \phi^n(x)}(x, y)\right| 
		&\leq \int_U \left|p^n(x,u)[y] - p(x,u)[y] \right| \, \phi^n(x) (\dd u) \\
		&\leq \sup_{u \in U} \left|p^n(x,u)[y] - p(x,u)[y] \right| \to 0.
	\end{align*}
	
    To handle the second term, note that by \eqref{eq:Pf}, the map $u \mapsto p(x, u)[y]$ is $L$-Lipschitz for all $x, y \in \Dc$. By the Kantorovich-Rubinstein duality and the Wasserstein-1 convergence $\phi^n(x) \to \phi(x)$,
    \begin{align*}
     \left|f^{\bP, \phi^n(x)}(x, y) -  f^{\bP, \phi(x)}(x, y) \right| &= \left| \int_U p(x, u)[y] \, \phi^n(x) (\dd u) - \int_U p(x, u)[y] \, \phi(x) (\dd u) \right| \\
     &\leq L \cdot d_{W_1}(\phi^n(x), \phi(x)) \to 0.
    \end{align*}
	
	Combining the above two results proves the entrywise convergence $\Phi(\bP^n, \phi^n) \to \Phi(\bP, \phi)$. In fact, because the state space $\Dc$ is finite, entrywise convergence is equivalent to convergence in any matrix norm. 
\end{proof}

In Subsection \ref{sub:MFG_equ}, we show that a stationary equilibrium $\pi^*$ to the MFG can be equivalently characterized by the one-shot optimality condition in \eqref{eq:pi_op_one} (see Proposition~\ref{prop:one_shot_eps} for the $\varepsilon$-equilibrium). Our next goal is to study the utility function of one-shot perturbed strategies, $\pi^i \otimes_1 \pi^*$, which applies $\pi^i$ in the first period and switches to $\pi^*$ from the second period onward. Applying the DPP in \eqref{eq:V}, we obtain 
\begin{align}
	\label{eq:one-shot-mf-auxiliary}
	V^{\pi^i \otimes_1 \pi^*,\pi^*}(x, v) = \inf_{\bP \in \Pf(v)} \mathbb{E}^{\bP}_{x, v} \left[
	r(x,\pi^i, v, X_1^i) + \rho \cdot V^{\pi^*,\pi^*}(X_1^i, \mu_1) \right].
\end{align}
For every $\bP := \bP(v) = \{p(x, u, v)\}_{(x, u) \in \Dc \times U}$, the probability that $X_1^i = y$ is equal to $f^{\bP(v), \pi^i(x,v)}(x, y)$ and the population distribution in the next period is given by $\mu_1 := \mu_1^{\bP, \pi^*}(v) = \Phi(\bP(v), \pi^*(v)) \cdot v$; hereafter, we may add an argument $v$ to emphasize the dependence on the population distribution $v$, and the same rule also applies to other variables. Using these results, we rewrite the expectation in \eqref{eq:one-shot-mf-auxiliary} as 
\begin{equation}
	\label{eq:one-shot-mf-expanded}
	\begin{split}
		V^{\pi^i \otimes_1 \pi^*,\pi^*}(x, v) = \inf_{\bP \in \Pf(v)} \sum_{y \in \Dc} 
		\Bigg[&\int_U r(x,u,v, y) p(x, u)[y] \, \pi^i(\dd u|x,v) \\
		&+ \rho \, \int_U p(x, u)[y] \, \pi^i(\dd u|x,v)  \cdot V^{\pi^*,\pi^*} \left(y, \mu_1^{\bP, \pi^*}(v) \right) \Bigg].
	\end{split}
\end{equation}
Inside the square bracket, the first term is the \emph{immediate reward}, and the second term is referred to as the \emph{continuation value}, when the future state is in $y$. For convenience, we denote 
\begin{equation}
	\label{eq:auxiliary-mf-W}
	W^{\bP, \pi^*} (x,v,y) := V^{\pi^*,\pi^*} \left(y, \mu_1^{\bP, \pi^*}(v) \right),
\end{equation}
which can be interpreted as the equilibrium value function for an agent who is in state $y$ when the population distribution is $\mu_1^{\bP, \pi^*}(v)$.

We now derive the analogous one-shot utility for the $N$-agent game. 
The structure parallels the MFG analysis but accounts for the 
stochastic evolution of the empirical distribution.
For the $N$-agent game, assume Agent $i$ applies a one-shot deviation $\pi^{N,i} \otimes_1 \pi^{N,*}$, where $\pi^{N,i} \in \Pi$ is used in period 0, and $\pi^{N,*} \in \Pi$ is the common strategy thereafter. Using the DPP analogous to \eqref{eq:V}, we have
\begin{align}
	\label{eq:one-shot-nagent-auxiliary}
	V^{N,\pi^{N,i} \otimes_1 \pi^{N,*},\pi^{N,*}}(x, v) = \inf_{\bP \in \Pf(v)} \mathbb{E}^{\bP}_{x, v} \left[
	r(x,\pi^{N,i}, v, X_1^{N,i}) + \rho \cdot V^{N,\pi^{N,*},\pi^{N,*}}(X_1^{N,i}, \mu_1^N) \right],
\end{align}
where $X_1^{N,i}$ is Agent $i$'s state at time 1, and $\mu_1^N$ is the \emph{random} empirical distribution defined in \eqref{eq:mu_N}.

The key difference from the MFG case lies in the continuation value. In the MFG, the next-period distribution $\mu_1^{\bP, \pi^*}(v)$ is deterministic (given $\bP$ and $\pi^*$); however, in the $N$-agent game, $\mu_1^N$ is random because it depends on the independent transitions of all $N$ agents. To make this explicit, we expand the expectation in \eqref{eq:one-shot-nagent-auxiliary}.
For every $\bP := \bP(v) = \{p(x, u, v)\}_{(x, u) \in \Dc \times U} \in \Pf(v)$, Agent $i$'s transition to state $y$ occurs with probability $f^{\bP(v), \pi^{N,i}(x,v)}(x, y)$ (as in the MFG case). However, the continuation value now involves:
\begin{align}
	\label{eq:continuation-nagent}
	\mathbb{E}^{\bP}_{x,v}\left[V^{N,\pi^{N,*},\pi^{N,*}}(X_1^{N,i}, \mu_1^N)\right] 
	= \sum_{y \in \Dc} f^{\bP(v), \pi^{N,i}(x,v)}(x, y) \cdot \mathbb{E}^{\bP}_{x,v}\left[V^{N,\pi^{N,*},\pi^{N,*}}(y, \mu_1^N) \,\Big|\, X_1^{N,i} = y\right].
\end{align}

The conditional expectation in \eqref{eq:continuation-nagent} captures the randomness of the empirical distribution $\mu_1^N$, given that Agent $i$ reaches state $y$. Given the initial configuration $(X_0^{N,i} = x, \mu_0^N = v)$, we know the initial states of all $N$ agents. Conditional on $X_1^{N,i} = y$, Agent $i$'s transition is fixed, but the other $N-1$ agents transition independently according to $\bP$ and $\pi^{N,*}$. The empirical distribution $\mu_1^N$ is then determined by the collective outcomes of these $N-1$ independent transitions plus Agent $i$'s fixed state $y$.

To emphasize this conditioning, we introduce the notation $\mu_1^{N,y}$ for the conditional empirical distribution
\begin{align}
	\label{eq:mu1N-y-decomposition}
	\mu_1^{N,y} := \frac{1}{N}\left(\delta_y + \sum_{j \neq i} \delta_{X_1^{N,j}}\right),
\end{align}
where the sum is over the $N-1$ other agents. This notation makes explicit that Agent $i$ is in state $y$ at time 1, while the other agents' states remain random. Note that although the $\frac{1}{N}\delta_y$ term is deterministic and has magnitude $O(1/N)$, it is important to track this term carefully in the convergence analysis.

Combining \eqref{eq:one-shot-nagent-auxiliary} and \eqref{eq:continuation-nagent}, we obtain the expanded form
\begin{equation}
	\label{eq:one-shot-nagent-expanded}
	\begin{split}
		V^{N,\pi^{N,i} \otimes_1 \pi^{N,*},\pi^{N,*}}(x, v) = \inf_{\bP \in \Pf(v)} \sum_{y \in \Dc} 
		\Bigg[&\int_U r(x,u,v, y) p(x, u)[y] \, \pi^{N,i}(\dd u|x,v) \\
		&+ \rho \, \int_U p(x, u)[y] \, \pi^{N,i}(\dd u|x,v) \cdot W^{N,\bP, \pi^{N,*}}(x,v,y) \Bigg],
	\end{split}
\end{equation}
where the \emph{$N$-agent auxiliary function $W^{N,\bP, \pi^{N,*}}(x,v,y)$ } is 
\begin{equation}
	\label{eq:auxiliary-nagent-W}
	W^{N,\bP, \pi^{N,*}}(x,v,y) := \mathbb{E}^{\bP}_{x,v}\left[V^{N,\pi^{N,*},\pi^{N,*}}(y, \mu_1^{N,y}) \,\Big|\, X_1^{N,i} = y\right].
\end{equation}

\begin{remark}[Comparison: MFG vs $N$-agent auxiliary functions]
	Comparing \eqref{eq:auxiliary-mf-W} and \eqref{eq:auxiliary-nagent-W} reveals the fundamental difference. In the MFG, $W^{\bP, \pi^*}(x,v,y) = V^{\pi^*,\pi^*}(y, \mu_1^{\bP, \pi^*}(v))$ is a \emph{deterministic} function, because $\mu_1^{\bP, \pi^*}(v)$ is fully determined by $\bP$, $\pi^*$, and $v$. In contrast, for the $N$-agent game, $W^{N,\bP, \pi^{N,*}}(x,v,y) = \mathbb{E}[V^{N,\pi^{N,*},\pi^{N,*}}(y, \mu_1^{N,y})]$ involves an \emph{expectation} over the conditional random empirical distribution $\mu_1^{N,y}$. This distinction is precisely what makes the convergence analysis nontrivial: we must show that as $N \to \infty$, the random distribution $\mu_1^{N,y}$ concentrates around the deterministic limit $\mu_1^{\bP, \pi^*}(v)$, enabling $W^{N,\bP, \pi^{N,*}} \to W^{\bP, \pi^*}$.
\end{remark}

\begin{remark}[Structure preservation]
	Despite the difference in randomness, notice that the one-shot formulas \eqref{eq:one-shot-mf-expanded} and \eqref{eq:one-shot-nagent-expanded} have \emph{identical structure}
	$$V = \inf_{\bP \in \Pf(v)} \sum_{y \in \Dc} \left[\text{Immediate Reward} + \rho \cdot \text{Transition Probability} \cdot W(x,v,y)\right].$$
	The only difference lies in the auxiliary functions $W$ vs $W^N$. This structural similarity is crucial for the convergence proofs in Subsections \ref{subsec:proof-approximation} and \ref{subsec:proof-convergence}, as it allows us to focus on establishing
	$$W^{N,\bP, \pi^N}(x,v,y) \to W^{\bP, \pi^*}(x,v,y) \quad \text{(uniformly) as } N \to \infty.$$
\end{remark}

\subsection{Technical Lemmas for Main Results}

We establish several technical lemmas that are essential for proving the main convergence and existence results in Subsection \ref{sub:main}. Throughout this subsection, we work with a sequence $(\pi^N)_{N \in \mathbb{N}} \subset \Pi_N$ and $\pi^\infty \in \Pi_c$ satisfying the uniform convergence condition
\begin{equation}
	\label{eq:strategy-convergence}
	\lim_{N \to \infty} \sup_{(x,v) \in \Delta_N} d_{W_1}(\pi^N(x,v), \pi^\infty(x,v)) = 0.
\end{equation}

Additionally, for the pointwise convergence results, we utilize the fact that 
any distribution $v \in \Pc(\Dc)$ can be approximated by empirical distributions. 
Specifically, for any $(x,v) \in \Delta$ and $N \geq 2$, we construct 
$v(N,x) \in \Pc(\Dc)$ such that $(x, v(N,x)) \in \Delta_N$ and
\begin{equation}
	\label{eq:v-N-x-convergence}
	v(N,x) \to v \quad \text{as } N \to \infty.
\end{equation}
One explicit construction is as follows. Let $\tilde{v}_N \in \Pc(\Dc)$ be any 
distribution satisfying $(N-1)\tilde{v}_N \in \Nb^d$ (i.e., $\tilde{v}_N$ is an 
empirical distribution of $N-1$ agents) such that $\tilde{v}_N \to v$. Such 
sequences exist because empirical distributions are dense in $\Pc(\Dc)$. 
Then define
\[
v(N,x) := \frac{1}{N}\bigl(\delta_x + (N-1)\tilde{v}_N\bigr).
\]
By construction, $N \cdot v(N,x) = \delta_x + (N-1)\tilde{v}_N \in \Nb^d$ and 
$N \cdot v(N,x)[x] \geq 1 > 0$, so $(x, v(N,x)) \in \Delta_N$. Moreover, 
$v(N,x) \to v$ since $\tilde{v}_N \to v$.

\subsubsection{Regularity of the Value Functions}

\begin{lemma}[Continuity of the Mean-Field Value Function]
	\label{lem:value-continuity}
	Let $\pi^* \in \Pi_c$. Then, the value function of the MFG $V^{\pi^*,\pi^*}: \Delta \to \mathbb{R}$ is continuous.
\end{lemma}

\begin{proof}
	We show that the Bellman operator $\mathcal{T}^{\pi^*}: \mathcal{B}(\Delta) \to \mathcal{B}(\Delta)$ defined by
	$$(\mathcal{T}^{\pi^*} W)(x,v) = \inf_{\bP \in \Pf(v)} \mathbb{E}^{\bP}\left[r(x,\pi^*,v,X_1) + \rho W(X_1, \mu_1^{\bP,\pi^*}(v))\right]$$
	preserves continuity, then apply Proposition \ref{prop:V_same} to conclude that the unique fixed point $V^{\pi^*,\pi^*}$ is continuous.

    \textbf{Step 1: Bellman operator structure.}
    For $W \in C( \Delta)$, continuity on the compact set $\Delta$ implies that $W$ is bounded and uniformly continuous. Let $\|W\|_\infty := \sup_{(y,\mu) \in \Delta} |W(y,\mu)| < \infty$. We can write:
    $$(\mathcal{T}^{\pi^*} W)(x,v) = \inf_{\bP \in \Pf(v)} F(x,v,\bP),$$
    where
    $$F(x,v,\bP) := \int_U \sum_{y \in \Dc} p(x,u)[y] \left[r(x,u,v,y) + \rho W(y, \Phi(\bP, \pi^*(v)) \cdot v)\right] \pi^*(\dd u|x,v).$$

	\textbf{Step 2: Constraint correspondence.}
	The correspondence $(x,v) \mapsto \Pf(v)$ is continuous with non-empty compact convex values by Assumption \ref{ass:uncertainty}. Note that the correspondence is independent of $x$, which simplifies the analysis.

    \textbf{Step 3: Joint continuity of $F$.}
    To apply Berge's maximum theorem, we verify that $F(x,v,\bP)$ is jointly continuous in $(x,v,\bP)$. Consider a sequence $(x^n, v^n, \bP^n) \to (x, v, \bP)$ with $\bP^n \in \Pf(v^n)$ and $\bP \in \Pf(v)$. Since $\Dc$ is discrete, we may assume $x^n = x$ for all $n$. We decompose:
    $$F(x,v^n,\bP^n) - F(x,v,\bP) = \big(F(x,v^n,\bP^n) - F(x,v^n,\bP)\big) + \big(F(x,v^n,\bP) - F(x,v,\bP)\big).$$

    \textit{Analysis of the first term.}
    Using the boundedness of $r$ and $W$, we have
    \begin{align*}
    |F(x,v^n,\bP^n) - F(x,v^n,\bP)| 
    &\leq (C_r + \rho \|W\|_\infty ) \sup_{u,y} |p^n(x,u)[y] - p(x,u)[y]| \\
    &\quad + \sup_y \left|W(y, \Phi(\bP^n, \pi^*(v^n)) \cdot v^n) - W(y, \Phi(\bP, \pi^*(v^n)) \cdot v^n)\right|.
    \end{align*}
    The first term vanishes by uniform convergence $\bP^n \to \bP$. For the second term, we use the triangle inequality:
    \begin{align*}
    \left\|\Phi(\bP^n, \pi^*(v^n)) - \Phi(\bP, \pi^*(v^n))\right\| 
    \leq \left\|\Phi(\bP^n, \pi^*(v^n)) - \Phi(\bP, \pi^*(v))\right\| + \left\|\Phi(\bP, \pi^*(v)) - \Phi(\bP, \pi^*(v^n))\right\|.
    \end{align*}
    Since $\pi^* \in \Pi_c$, we have $\pi^*(v^n) \to \pi^*(v)$ pointwise. By Lemma~\ref{lem:transition-continuity}, the first term vanishes using $(\bP^n, \pi^*(v^n)) \to (\bP, \pi^*(v))$, and the second term vanishes using $\pi^*(v^n) \to \pi^*(v)$ with $\bP$ fixed. The uniform continuity of $W$ then yields the convergence.

    \textit{Analysis of the second term.}
    Define
    $$h^n(u,y) := p(x,u)[y]\left[r(x,u,v^n,y) + \rho W(y, \Phi(\bP, \pi^*(v^n)) \cdot v^n)\right],$$
    $$h(u,y) := p(x,u)[y]\left[r(x,u,v,y) + \rho W(y, \Phi(\bP, \pi^*(v)) \cdot v)\right].$$
    Then
    \begin{align*}
    F(x,v^n,\bP) - F(x,v,\bP) &= \int_U \sum_y h^n(u,y) \, \pi^*(\dd u|x,v^n) - \int_U \sum_y h(u,y) \, \pi^*(\dd u|x,v) \\
    &= \int_U \sum_y \big(h^n(u,y) - h(u,y)\big) \, \pi^*(\dd u|x,v^n) \\
    &\quad + \int_U \sum_y h(u,y) \, \big(\pi^*(\dd u|x,v^n) - \pi^*(\dd u|x,v)\big).
    \end{align*}

    For the first integral, Lemma~\ref{lem:transition-continuity} gives $\Phi(\bP, \pi^*(v^n)) \to \Phi(\bP, \pi^*(v))$. Combined with continuity of $r$ and $W$, we obtain $\sup_{u,y} |h^n(u,y) - h(u,y)| \to 0$, hence the integral vanishes.

    For the second integral, define $g(u) := \sum_y h(u,y)$. By \eqref{eq:Pf} and Assumption~\ref{ass:r}, $g$ is Lipschitz in $u$ with some constant $L_g$ depending on $L$, $C_r$, $L_r$, and $\|W\|_\infty$. Since $\pi^* \in \Pi_c$, we have $d_{W_1}(\pi^*(x,v^n), \pi^*(x,v)) \to 0$. By the Kantorovich-Rubinstein duality,
    $$\left|\int_U g(u) \, \pi^*(\dd u|x,v^n) - \int_U g(u) \, \pi^*(\dd u|x,v)\right| \leq L_g \cdot d_{W_1}(\pi^*(x,v^n), \pi^*(x,v)) \to 0.$$
    
    This completes the proof of joint continuity.
   
	\textbf{Step 4: Application of Berge's theorem.}
	By Berge's maximum theorem, the function $(x,v) \mapsto \inf_{\bP \in \Pf(v)} F(x,v,\bP)$ is continuous. Therefore, $\mathcal{T}^{\pi^*}$ maps $C(\Delta)$ to $C(\Delta)$.
	
	By Proposition \ref{prop:V_same}, the value function $V^{\pi^*,\pi^*}$ is the unique bounded function satisfying the Bellman equation $V = \mathcal{T}^{\pi^*} V$. Because $\mathcal{T}^{\pi^*}$ maps the closed subspace $C(\Delta) \subset \mathcal{B}(\Delta)$ to itself, and the fixed point in $\mathcal{B}(\Delta)$ is unique, we conclude that $V^{\pi^*,\pi^*} \in C(\Delta)$ is continuous.
\end{proof}

\begin{remark}
	The proof crucially uses the fact that $\pi^* \in \Pi_c$, which ensures that $\pi^*(v^n) \to \pi^*(v)$ as measures in the Wasserstein distance. This continuity, combined with Lemma \ref{lem:transition-continuity}, allows us to track how the next-period population distribution $\Phi(\bP, \pi^*(v)) \cdot v$ varies with $(v,\bP)$, which is essential for establishing the joint continuity of $F$.
\end{remark}

The preceding lemma establishes continuity of the mean-field value function with respect to the state variables $(x,v)$ for a fixed strategy $\pi^*$. To prove the existence of $N$-agent equilibria, we require a different type of continuity: the value function $V^{N,\pi,\pi}$ must be continuous with respect to the strategy $\pi$ itself, uniformly over all state pairs $(x,v) \in \Delta_N$. This uniform continuity in strategy space is essential for applying  Kakutani's fixed-point theorem to establish existence.

\begin{lemma}[Continuity of the $N$-agent Value Function in Strategy]
    \label{lem:n-agent-value-continuity}
    For any $N \geq 2$, the value function of the $N$-agent game $V^{N,\pi,\pi}(x,v)$ is continuous in $\pi$ uniformly over $(x,v) \in \Delta_N$. More precisely, if $\pi^n \to \pi$ in $\Pi_N$ (i.e., $d_{W_1}(\pi^n(x,v), \pi(x,v)) \to 0$ for each $(x,v) \in \Delta_N$), then
    $$\|V^{N,\pi^n,\pi^n} - V^{N,\pi,\pi}\|_{\infty,N} \to 0,$$
    where $\|F\|_{\infty,N} := \max_{(x,v) \in \Delta_N} |F(x,v)|$ for any $F: \Delta_N \to \mathbb{R}$.
\end{lemma}

\begin{proof}
    Because $\Delta_N$ is finite, pointwise convergence $\pi^n \to \pi$ implies
    $$\epsilon_n := \max_{(x,v) \in \Delta_N} d_{W_1}(\pi^n(x,v), \pi(x,v)) \to 0.$$
    
    \textbf{Step 1: Operator convergence.}
    \label{lem:operator-convergence-step1} 
    For any $F: \Delta_N \to \mathbb{R}$ with $\|F\|_{\infty,N} \leq M$, we show $\|\mathcal{T}^{N,\pi^n} F - \mathcal{T}^{N,\pi} F\|_{\infty,N} \to 0$.
    
    Fix $(x,v) \in \Delta_N$ and $\bP \in \Pf(v)$. For $\sigma \in \{\pi^n, \pi\}$, define
    $$G^\sigma(x,v,\bP) := \int_U \sum_{y \in \Dc} p(x,u)[y] \left[r(x,u,v,y) + \rho \, H^\sigma(y)\right] \sigma(\dd u|x,v),$$
    where $H^\sigma(y) := \mathbb{E}^\sigma[F(y, \mu_1^{N,y})]$ is the expected continuation value when Agent $i$ reaches state $y$, with the expectation taken over the transitions of the other $N-1$ agents under strategy $\sigma$. Then
    \begin{align*}
        |G^{\pi^n} - G^\pi| 
        &\leq \left|\int_U g^{\pi^n}(u) \, \pi^n(\dd u) - \int_U g^{\pi^n}(u) \, \pi(\dd u)\right| + \left|\int_U g^{\pi^n}(u) \, \pi(\dd u) - \int_U g^\pi(u) \, \pi(\dd u)\right|,
    \end{align*}
    where $g^\sigma(u) := \sum_{y} p(x,u)[y][r(x,u,v,y) + \rho H^\sigma(y)]$.
    
    \textit{First term.} By \eqref{eq:Pf} and Assumption \ref{ass:r}, $g^{\pi^n}$ is Lipschitz in $u$ with constant $L_g$ depending on $L$, $C_r$, $L_r$, and $M$. By the Kantorovich-Rubinstein duality,
    $$\left|\int_U g^{\pi^n}(u) \, \pi^n(\dd u|x,v) - \int_U g^{\pi^n}(u) \, \pi(\dd u|x,v)\right| \leq L_g \cdot d_{W_1}(\pi^n(x,v), \pi(x,v)) \leq L_g \epsilon_n.$$
    
    \textit{Second term.} We bound $|H^{\pi^n}(y) - H^\pi(y)|$ for each $y \in \Dc$. Under strategy $\sigma$, the empirical distribution of the other $N-1$ agents at time 1 is
    $$\bar{\mu}^\sigma := \frac{1}{N-1}\sum_{j \neq i} \delta_{X_1^{\sigma,j}},$$
    where $X_1^{\sigma,j}$ denotes Agent $j$'s state at time 1 under strategy $\sigma$. Then $\mu_1^{N,y}$ under strategy $\sigma$ equals $\frac{1}{N}\delta_y + \frac{N-1}{N}\bar{\mu}^\sigma$.
    
    We construct a coupling of $(X_1^{\pi^n,j}, X_1^{\pi,j})_{j \neq i}$: for each $j \neq i$ with initial state $x_j$ (determined by $v$), let $\gamma_j$ be an optimal coupling of $(\pi^n(x_j,v), \pi(x_j,v))$. Sample $(u_j^n, u_j) \sim \gamma_j$, then sample $(X_1^{\pi^n,j}, X_1^{\pi,j})$ from an optimal coupling of $(p(x_j,u_j^n), p(x_j,u_j))$. Under this joint coupling, we have
    \begin{align}
        \label{eq:coupling-transition}
        \mathbb{E}[d_\Dc(X_1^{\pi^n,j}, X_1^{\pi,j})] \leq L \cdot d_{W_1}(\pi^n(x_j,v), \pi(x_j,v)) \leq L\epsilon_n,
    \end{align}
    where the first inequality follows from the Lipschitz property in \eqref{eq:Pf}.
    
    Now, under this coupling, $(\bar{\mu}^{\pi^n}, \bar{\mu}^\pi)$ are jointly defined on the same probability space. Because $\Delta_N$ is finite, we can bound
    \begin{align*}
        |H^{\pi^n}(y) - H^\pi(y)| 
        &= \left|\mathbb{E}\left[F\left(y, \tfrac{1}{N}\delta_y + \tfrac{N-1}{N}\bar{\mu}^{\pi^n}\right)\right] - \mathbb{E}\left[F\left(y, \tfrac{1}{N}\delta_y + \tfrac{N-1}{N}\bar{\mu}^\pi\right)\right]\right| \\
        &\leq 2M \cdot \mathbb{P}(\bar{\mu}^{\pi^n} \neq \bar{\mu}^\pi) \\
        &\leq 2M \cdot \mathbb{P}\left(\exists j \neq i: X_1^{\pi^n,j} \neq X_1^{\pi,j}\right) \\
        &\leq 2M \sum_{j \neq i} \mathbb{P}(X_1^{\pi^n,j} \neq X_1^{\pi,j}).
    \end{align*}
    
    For each $j \neq i$, using \eqref{eq:coupling-transition} with the discrete metric $d_\Dc(x,y) = \mathbf{1}_{x \neq y}$,
    $$\mathbb{P}(X_1^{\pi^n,j} \neq X_1^{\pi,j}) = \mathbb{E}[d_\Dc(X_1^{\pi^n,j}, X_1^{\pi,j})] \leq L\epsilon_n.$$
    
    Hence $|H^{\pi^n}(y) - H^\pi(y)| \leq 2M(N-1)L\epsilon_n$, and the second term is bounded by $\rho \cdot 2M(N-1)L\epsilon_n$.
    
    Combining both terms, taking infimum over $\bP \in \Pf(v)$ and maximum over $(x,v)$ yields
    $$\|\mathcal{T}^{N,\pi^n} F - \mathcal{T}^{N,\pi} F\|_{\infty,N} \leq (L_g + 2\rho M(N-1)L)\epsilon_n \to 0.$$

    \textbf{Step 2: Fixed point convergence.}
    Let $V^n := V^{N,\pi^n,\pi^n}$ and $V := V^{N,\pi,\pi}$. By the boundedness of the 
    reward function (Assumption \ref{ass:r}), both value functions satisfy
    $$\|V^n\|_{\infty,N}, \|V\|_{\infty,N} \leq \frac{C_r}{1-\rho} =: M_V.$$
    Taking $M = M_V$ in Step 1, we obtain constants $L_g = L_g(L, C_r, L_r, M_V)$ and 
    $C_N := L_g + 2\rho M_V(N-1)L$ such that for any $F: \Delta_N \to \mathbb{R}$ with 
    $\|F\|_{\infty,N} \leq M_V$,
    $$\|\mathcal{T}^{N,\pi^n} F - \mathcal{T}^{N,\pi} F\|_{\infty,N} \leq C_N \epsilon_n.$$

    Analogous to Proposition \ref{prop:V_same}, $V^n$ and $V$ are the unique fixed points 
    of $\mathcal{T}^{N,\pi^n}$ and $\mathcal{T}^{N,\pi}$ respectively. Using the fixed 
    point equations and the contraction property:
    \begin{align*}
    \|V^n - V\|_{\infty,N} 
    &= \|\mathcal{T}^{N,\pi^n}(V^n) - \mathcal{T}^{N,\pi}(V)\|_{\infty,N} \\
    &\leq \|\mathcal{T}^{N,\pi^n}(V^n) - \mathcal{T}^{N,\pi^n}(V)\|_{\infty,N} + 
      \|\mathcal{T}^{N,\pi^n}(V) - \mathcal{T}^{N,\pi}(V)\|_{\infty,N} \\
    &\leq \rho \|V^n - V\|_{\infty,N} + C_N \epsilon_n,
    \end{align*}
    where the last inequality uses $\|V\|_{\infty,N} \leq M_V$. 
    Rearranging yields $\|V^n - V\|_{\infty,N} \leq \frac{C_N}{1-\rho} \epsilon_n \to 0$.
\end{proof}

\subsubsection{Concentration of Distributions}

The next two lemmas establish that empirical distributions concentrate around their mean-field limits as the population size grows. The first lemma provides uniform concentration across all finite $N$-agent games, which is essential for proving Theorem~\ref{thm:mf-limit} (approximation from MFG to $N$-agent). The second lemma establishes pointwise concentration along sequences of approximating distributions, which will be used in proving Theorem~\ref{thm:convergence-mf} (limits are MFG equilibria).

Because $\mathcal{P}(\Dc)$ is a finite-dimensional space, all norms on it are equivalent. For convenience, we work with the $\ell^1$ norm: for two distributions $v, w \in \mathcal{P}(\Dc)$,
$$|v - w| = \sum_{x \in \Dc} |v[x] - w[x]|.$$

\begin{lemma}[Uniform Concentration of Conditional Empirical Distributions]
	\label{lem:distribution-convergence}
	For each fixed $y \in \Dc$, let $\mu_1^{\bP, \pi^\infty}(v) = \Phi(\bP, \pi^\infty(v)) \cdot v$ denote the mean-field limit distribution. Then for all $\varepsilon > 0$,
	$$\lim_{N \to \infty} \sup_{(x,v) \in \Delta_N} \sup_{\bP \in \Pf(v)} \sup_{y \in \Dc} \mathbb{P}\left(|\mu_1^{N,y} - \mu_1^{\bP, \pi^\infty}(v)| \geq \varepsilon\right) = 0,$$
	where $\mu_1^{N,y}$ is the conditional empirical distribution defined by \eqref{eq:mu1N-y-decomposition}.
\end{lemma}

\begin{proof}
    Conditioning on $(X_0^{N,i} = x, \mu_0^N = v, X_1^{N,i} = y)$, the other $N-1$ agents have initial distribution $(Nv - \delta_x)/(N-1)$ and transition independently. Define $\bar{\mu} := \frac{1}{N-1}\sum_{j \neq i} \delta_{X_1^{N,j}}$. We decompose:
    \begin{align*}
        \mu_1^{N,y} - \mu_1^{\bP, \pi^\infty}(v) 
        &= \frac{1}{N}\delta_y + \frac{N-1}{N}\bar{\mu} - v \cdot \Phi(\bP, \pi^\infty(v)) \\
        &= \underbrace{\frac{1}{N}\delta_y}_{\text{(1)}} 
        + \frac{N-1}{N}\underbrace{(\bar{\mu} - \mathbb{E}[\bar{\mu}])}_{\text{(2)}} 
        + \underbrace{\frac{N-1}{N}\mathbb{E}[\bar{\mu}] - v \cdot \Phi(\bP, \pi^N(v))}_{\text{(3)}} \\
        &\quad + \underbrace{v \cdot (\Phi(\bP, \pi^N(v)) - \Phi(\bP, \pi^\infty(v)))}_{\text{(4)}}.
    \end{align*}
    
    \textbf{Term (1):} $|\frac{1}{N}\delta_y| = \frac{1}{N} \to 0$ uniformly.
    
    \textbf{Term (3):} From $\mathbb{E}[\bar{\mu}] = \frac{Nv - \delta_x}{N-1} \cdot \Phi(\bP, \pi^N(v))$, we have
    $$\frac{N-1}{N}\mathbb{E}[\bar{\mu}] - v \cdot \Phi(\bP, \pi^N(v)) = -\frac{1}{N}\delta_x \cdot \Phi(\bP, \pi^N(v)),$$
    whose norm is $\frac{1}{N}$, which vanishes uniformly as $N\rightarrow \infty$ (since rows of $\Phi$ have $\ell^1$ norm 1).
    
    \textbf{Term (4):} For each entry $(y',z) \in \Dc \times \Dc$, the map $u \mapsto p(y',u)[z]$ is $L$-Lipschitz by \eqref{eq:Pf}. By the Kantorovich-Rubinstein duality,
    $$|\Phi(\bP, \pi^N(v))[y',z] - \Phi(\bP, \pi^\infty(v))[y',z]| \leq L \cdot d_{W_1}(\pi^N(y',v), \pi^\infty(y',v)).$$
    Hence
    $$|v \cdot (\Phi(\bP, \pi^N(v)) - \Phi(\bP, \pi^\infty(v)))| \leq Ld \cdot \sup_{(x',v') \in \Delta_N} d_{W_1}(\pi^N(x',v'), \pi^\infty(x',v')) \to 0$$
    uniformly by \eqref{eq:strategy-convergence}.
    
    \textbf{Term (2):} 
    The conditioning on $X_1^{N,i} = y$ fixes Agent $i$'s transition but does not affect the transitions of other agents, which remain conditionally independent given their initial states. For each $z \in \Dc$, Hoeffding's inequality gives
    $$\mathbb{P}\left(\left|\frac{1}{N-1}\sum_{j \neq i} \mathbf{1}_{X_1^{N,j} = z} - \mathbb{E}[\frac{1}{N-1}\sum_{j \neq i} \mathbf{1}_{X_1^{N,j} = z}]\right| \geq \frac{\delta}{d}\right) \leq 2\exp\left(-\frac{2(N-1)\delta^2}{d^2}\right).$$
    By union bound over $d$ states, $\mathbb{P}(|\bar{\mu} - \mathbb{E}[\bar{\mu}]| \geq \delta) \leq 2d\exp(-2(N-1)\delta^2/d^2) \to 0$ uniformly.
    
    Combining all terms with the triangle inequality completes the proof.
\end{proof}

The preceding lemma establishes uniform concentration across all states in $\Delta_N$. For the proof of Theorem~\ref{thm:convergence-mf}, we also require a pointwise version along sequences of approximating distributions. The following lemma shows that empirical distributions concentrate around the mean-field limit along the approximating sequences $v(N,x)$ defined by \eqref{eq:v-N-x-convergence}.

\begin{lemma}[Pointwise Concentration of Empirical Distributions]
	\label{lem:distribution-convergence-pointwise}
	For each fixed $y \in \Dc$ and $(x,v) \in \Delta$, if $\bP_N \to \bP$ with $\bP_N \in \Pf(v(N,x))$ 
	and $\bP \in \Pf(v)$, then 
	$$\mu_1^{N,y} \xrightarrow{\mathbb{P}} \mu_1^{\bP,\pi^\infty}(v),$$
	where $\mu_1^{N,y}$ is evaluated at the state $(x, v(N,x))$ under kernel $\bP_N$ and strategy $\pi^N$.
\end{lemma}

\begin{proof}
    The proof follows the same decomposition as Lemma \ref{lem:distribution-convergence}. Define $\bar{\mu} := \frac{1}{N-1}\sum_{j \neq i} \delta_{X_1^{N,j}}$. We decompose:
    \begin{align*}
        \mu_1^{N,y} - \mu_1^{\bP, \pi^\infty}(v) 
        &= \underbrace{\frac{1}{N}\delta_y}_{\text{(1)}} 
        + \frac{N-1}{N}\underbrace{(\bar{\mu} - \mathbb{E}[\bar{\mu}])}_{\text{(2)}} 
        + \underbrace{\frac{N-1}{N}\mathbb{E}[\bar{\mu}] - v(N,x) \cdot \Phi(\bP_N, \pi^N(v(N,x)))}_{\text{(3)}} \\
        &\quad + \underbrace{v(N,x) \cdot \Phi(\bP_N, \pi^N(v(N,x))) - v \cdot \Phi(\bP, \pi^\infty(v))}_{\text{(4)}}.
    \end{align*}
    
    \textbf{Terms (1)--(3):} Following the analysis in Lemma \ref{lem:distribution-convergence}, Term (1) equals $\frac{1}{N} \to 0$, Term (3) equals $\frac{1}{N} \to 0$, and Term (2) vanishes in probability by Hoeffding's inequality.
    
    \textbf{Term (4):} We bound
    \begin{align*}
        &|v(N,x) \cdot \Phi(\bP_N, \pi^N(v(N,x))) - v \cdot \Phi(\bP, \pi^\infty(v))| \\
        &\leq |v(N,x) - v| \cdot \|\Phi(\bP_N, \pi^N(v(N,x)))\|_1 
        + \|\Phi(\bP_N, \pi^N(v(N,x))) - \Phi(\bP, \pi^\infty(v))\|_1,
    \end{align*}
    where $\|\cdot\|_1$ denotes the matrix $\ell^1$ norm. Because each row of $\Phi$ sums to 1, we have $$\|\Phi(\bP_N, \pi^N(v(N,x)))\|_1 \leq d.$$ The first term vanishes by \eqref{eq:v-N-x-convergence}. 
    For the second term, by the triangle inequality:
    \begin{align*}
    &d_{W_1}(\pi^N(y', v(N,x)), \pi^\infty(y', v)) \\
    & \leq d_{W_1}(\pi^N(y', v(N,x)), \pi^\infty(y', v(N,x))) + d_{W_1}(\pi^\infty(y', v(N,x)), \pi^\infty(y', v)).
    \end{align*}
    The first term vanishes by \eqref{eq:strategy-convergence}, and the second by the continuity of $\pi^\infty \in \Pi_c$ and $v(N,x) \to v$.
    
    Combining all terms establishes convergence in probability.
\end{proof}

\subsubsection{Convergence Results of the Value Functions and Auxiliary Functions}

Having established the concentration of empirical distributions, we now prove the convergence of the value functions and auxiliary functions. Recall from Subsection \ref{subsec:pre} that the auxiliary functions are defined in terms of value functions: the mean-field auxiliary function is $W^{\bP, \pi^*}(x,v,y) := V^{\pi^*,\pi^*}(y, \mu_1^{\bP, \pi^*}(v))$, and the $N$-agent auxiliary function is $W^{N,\bP, \pi^N}(x,v,y) := \mathbb{E}^{\bP}_{x,v}[V^{N,\pi^N,\pi^N}(y, \mu_1^{N,y})]$. Because the auxiliary functions are expressed in terms of value functions, we establish their convergence simultaneously through a coupled contraction argument.

\begin{lemma}[Convergence of the Value Functions and Auxiliary Functions]
	\label{lem:value-auxiliary-uniform}
    The following convergence results hold
	\begin{align}
		&\lim_{N \to \infty} \sup_{(x,v) \in \Delta_N} |V^{N,\pi^N,\pi^N}(x,v) - V^{\pi^\infty,\pi^\infty}(x,v)| = 0, \label{eq:value-convergence} \\
		&\lim_{N \to \infty} \sup_{\substack{(x,v) \in \Delta_N \\ y \in \Dc, \bP \in \Pf(v)}} |W^{N,\bP,\pi^N}(x,v,y) - W^{\bP,\pi^\infty}(x,v,y)| = 0. \label{eq:auxiliary-convergence}
	\end{align}
\end{lemma}

\begin{proof}
	Define 
	\begin{align*}
		E_N^V &:= \sup_{(x,v) \in \Delta_N} |V^{N,\pi^N,\pi^N}(x,v) - V^{\pi^\infty,\pi^\infty}(x,v)|, \\
		E_N^W &:= \sup_{\substack{(x,v) \in \Delta_N \\ y \in \Dc, \bP \in \Pf(v)}} |W^{N,\bP,\pi^N}(x,v,y) - W^{\bP,\pi^\infty}(x,v,y)|.
	\end{align*}
	We will establish that both $E_N^V$ and $E_N^W$ converge to 0 through a coupled analysis.
    
\textbf{Step 1: Relating $E_N^V$ to $E_N^W$.}
	Based on the Bellman equations, for any $(x,v) \in \Delta_N$:
	\begin{align*}
		&V^{\pi^\infty,\pi^\infty}(x,v) = \inf_{\bP \in \Pf(v)} \sum_{y \in \Dc} \int_U p(x,u)[y] \left[r(x,u,v,y) + \rho W^{\bP,\pi^\infty}(x,v,y)\right] \pi^\infty(\dd u|x,v), \\
		&V^{N,\pi^N,\pi^N}(x,v) = \inf_{\bP \in \Pf(v)} \sum_{y \in \Dc} \int_U p(x,u)[y] \left[r(x,u,v,y) + \rho W^{N,\bP,\pi^N}(x,v,y)\right] \pi^N(\dd u|x,v).
	\end{align*}
 Because both expressions involve infima over the same set $\Pf(v)$, we have
	\begin{align*}
		&|V^{N,\pi^N,\pi^N}(x,v) - V^{\pi^\infty,\pi^\infty}(x,v)| \\
		&\leq \sup_{\bP \in \Pf(v)} \Bigg|\sum_{y \in \Dc} \int_U p(x,u)[y] \left[r(x,u,v,y) + \rho W^{N,\bP,\pi^N}(x,v,y)\right] \pi^N(\dd u|x,v) \\
		&\qquad\qquad - \sum_{y \in \Dc} \int_U p(x,u)[y] \left[r(x,u,v,y) + \rho W^{\bP,\pi^\infty}(x,v,y)\right] \pi^\infty(\dd u|x,v)\Bigg|.
	\end{align*}
As such,
	\begin{align*}
		\text{RHS} &\leq \sup_{\bP \in \Pf(v)} \Bigg[\left|\sum_{y} \int_U p(x,u)[y] r(x,u,v,y) \left(\pi^N(\dd u|x,v) - \pi^\infty(\dd u|x,v)\right)\right| \\
		&\qquad + \rho \left|\sum_{y} \int_U p(x,u)[y] W^{N,\bP,\pi^N}(x,v,y) \pi^N(\dd u|x,v) - \sum_{y} \int_U p(x,u)[y] W^{\bP,\pi^\infty}(x,v,y) \pi^\infty(\dd u|x,v)\right|\Bigg].
	\end{align*}
	
    For the reward term, define $g_r(u) := \sum_y p(x,u)[y] r(x,u,v,y)$. By \eqref{eq:Pf} and Assumption \ref{ass:r}, $g_r$ is Lipschitz in $u$ with constant $L_g$ depending on $L$, $C_r$, and $L_r$. By the Kantorovich-Rubinstein duality,
    $$\left|\int_U g_r(u) \, \pi^N(\dd u|x,v) - \int_U g_r(u) \, \pi^\infty(\dd u|x,v)\right| \leq L_g \cdot d_{W_1}(\pi^N(x,v), \pi^\infty(x,v)).$$
	
	For the continuation term, using the triangle inequality,
	\begin{align*}
		&\left|\sum_{y} \int_U p(x,u)[y] W^{N,\bP,\pi^N}(x,v,y) \pi^N(\dd u|x,v) - \sum_{y} \int_U p(x,u)[y] W^{\bP,\pi^\infty}(x,v,y) \pi^\infty(\dd u|x,v)\right| \\
		&\leq \sum_{y} \int_U p(x,u)[y] |W^{N,\bP,\pi^N}(x,v,y) - W^{\bP,\pi^\infty}(x,v,y)| \pi^N(\dd u|x,v) \\
		&\quad + \left|\sum_{y} \int_U p(x,u)[y] W^{\bP,\pi^\infty}(x,v,y) \left(\pi^N(\dd u|x,v) - \pi^\infty(\dd u|x,v)\right)\right| \\
		&\leq E_N^W + L_W \cdot d_{W_1}(\pi^N(x,v), \pi^\infty(x,v)),
	\end{align*}
    where $L_W$ is the Lipschitz constant of $g_W(u) := \sum_y p(x,u)[y] W^{\bP,\pi^\infty}(x,v,y)$ with respect to $u$. Because $W^{\bP,\pi^\infty}(x,v,y)$ does not depend on $u$ and $|W^{\bP,\pi^\infty}| \leq M := \|V^{\pi^\infty,\pi^\infty}\|_\infty$, by \eqref{eq:Pf} we have $L_W \leq LM$.
	
    Taking supremum over $(x,v) \in \Delta_N$, we have
    \begin{equation}
        \label{eq:EV-bound}
        E_N^V \leq (L_g + \rho LM) \sup_{(x,v) \in \Delta_N} d_{W_1}(\pi^N(x,v), \pi^\infty(x,v)) + \rho E_N^W.
    \end{equation}
    
	\textbf{Step 2: Relating $E_N^W$ to $E_N^V$.}
	For any $(x,v) \in \Delta_N$, $y \in \Dc$, and $\bP \in \Pf(v)$:
	\begin{align*}
		&|W^{N,\bP,\pi^N}(x,v,y) - W^{\bP,\pi^\infty}(x,v,y)| \\
		&= \left|\mathbb{E}^{\bP}_{x,v}\left[V^{N,\pi^N,\pi^N}(y, \mu_1^{N,y})\right] - V^{\pi^\infty,\pi^\infty}(y, \mu_1^{\bP, \pi^\infty}(v))\right| \\
		&\leq \mathbb{E}^{\bP}_{x,v}\left[|V^{N,\pi^N,\pi^N}(y, \mu_1^{N,y}) - V^{\pi^\infty,\pi^\infty}(y, \mu_1^{N,y})|\right] \\
		&\quad + \mathbb{E}^{\bP}_{x,v}\left[|V^{\pi^\infty,\pi^\infty}(y, \mu_1^{N,y}) - V^{\pi^\infty,\pi^\infty}(y, \mu_1^{\bP, \pi^\infty}(v))|\right].
	\end{align*}
	
	For the first term, because $\mu_1^{N,y} \in \Delta_N$, we have
	$$\mathbb{E}^{\bP}_{x,v}\left[|V^{N,\pi^N,\pi^N}(y, \mu_1^{N,y}) - V^{\pi^\infty,\pi^\infty}(y, \mu_1^{N,y})|\right] \leq E_N^V.$$
	
	For the second term, let $\omega(\cdot)$ be the modulus of uniform continuity of $V^{\pi^\infty,\pi^\infty}$ (which exists by Lemma \ref{lem:value-continuity}). For any $\eta > 0$, define $A_\eta := \{|\mu_1^{N,y} - \mu_1^{\bP, \pi^\infty}(v)| \leq \eta\}$. Then
	\begin{align*}
\mathbb{E}^{\bP}_{x,v}\left[|V^{\pi^\infty,\pi^\infty}(y, \mu_1^{N,y}) - V^{\pi^\infty,\pi^\infty}(y, \mu_1^{\bP, \pi^\infty}(v))|\right] \leq \omega(\eta) + 2M \cdot \mathbb{P}(A_\eta^c).
	\end{align*}
	
	By Lemma \ref{lem:distribution-convergence}, for any fixed $\eta > 0$, we have
	$$\sup_{(x,v) \in \Delta_N} \sup_{\bP \in \Pf(v)} \sup_{y \in \Dc} \mathbb{P}(|\mu_1^{N,y} - \mu_1^{\bP, \pi^\infty}(v)| \geq \eta) \to 0 , \quad \text{as } N \to \infty.$$
	
	Thus
	\begin{equation}
		\label{eq:EW-bound}
		E_N^W \leq E_N^V + \omega(\eta) + o(1),
	\end{equation}
	where $o(1) \to 0$ as $N \to \infty$.
    
	\textbf{Step 3: Convergence via contraction.}
	Substituting \eqref{eq:EW-bound} into \eqref{eq:EV-bound} yields 
	\begin{align*}
        E_N^V \leq (L_g + \rho LM) \sup_{(x,v) \in \Delta_N} d_{W_1}(\pi^N(x,v), \pi^\infty(x,v)) + \rho[E_N^V + \omega(\eta) + o(1)] .
    \end{align*}
	
	Rearranging, we have
    $$(1 - \rho) E_N^V \leq (L_g + \rho LM) \sup_{(x,v) \in \Delta_N} d_{W_1}(\pi^N(x,v), \pi^\infty(x,v)) + \rho\omega(\eta) + o(1).$$

    For any $\varepsilon > 0$, first choose $\eta > 0$ such that $\omega(\eta) < (1-\rho)\varepsilon/2$. Then, by \eqref{eq:strategy-convergence}, choose $N_0$ such that for all $N \geq N_0$,
    $$(L_g + \rho LM) \sup_{(x,v) \in \Delta_N} d_{W_1}(\pi^N(x,v), \pi^\infty(x,v)) + o(1) < (1-\rho)^2\varepsilon/2.$$
    It follows that for $N \geq N_0$,
    $$(1-\rho)E_N^V < (1-\rho)^2\varepsilon/2 + \rho(1-\rho)\varepsilon/2 = (1-\rho)\varepsilon/2 < (1-\rho)\varepsilon,$$
    hence $E_N^V < \varepsilon$. Substituting into \eqref{eq:EW-bound} gives
    $$E_N^W \leq E_N^V + \omega(\eta) + o(1) < \varepsilon + (1-\rho)\varepsilon/2 + o(1) < 2\varepsilon$$
    for $N$ sufficiently large. Since $\varepsilon > 0$ is arbitrary, this establishes both \eqref{eq:value-convergence} and \eqref{eq:auxiliary-convergence}.
\end{proof}

For the proof of Theorem~\ref{thm:convergence-mf}, we also require a pointwise version of the convergence results on the auxiliary functions along approximating sequences.

\begin{lemma}[Pointwise Convergence of the Auxiliary Functions]
	\label{lem:auxiliary-convergence-pointwise}
	For each fixed $y \in \Dc$ and $(x,v) \in \Delta$, if $\bP_N \to \bP$ with $\bP_N \in \Pf(v(N,x))$ and $\bP \in \Pf(v)$, then:
	$$W^{N,\bP_N,\pi^N}(x,v(N,x),y) \to W^{\bP,\pi^\infty}(x,v,y) \quad \text{as } N \to \infty.$$
\end{lemma}

\begin{proof}
    Using the triangle inequality yields
    \begin{align*}
        &|W^{N,\bP_N,\pi^N}(x,v(N,x),y) - W^{\bP,\pi^\infty}(x,v,y)| \\
        &\leq \mathbb{E}^{\bP_N}_{x,v(N,x)}\left[|V^{N,\pi^N,\pi^N}(y, \mu_1^{N,y}) - V^{\pi^\infty,\pi^\infty}(y, \mu_1^{N,y})|\right] \\
        &\quad + \mathbb{E}^{\bP_N}_{x,v(N,x)}\left[|V^{\pi^\infty,\pi^\infty}(y, \mu_1^{N,y}) - V^{\pi^\infty,\pi^\infty}(y, \mu_1^{\bP, \pi^\infty}(v))|\right]. 
    \end{align*}
    
    For the first term, because $\mu_1^{N,y} \in \Delta_N$, Lemma \ref{lem:value-auxiliary-uniform} gives
    $$\mathbb{E}^{\bP_N}_{x,v(N,x)}\left[|V^{N,\pi^N,\pi^N}(y, \mu_1^{N,y}) - V^{\pi^\infty,\pi^\infty}(y, \mu_1^{N,y})|\right] \leq \sup_{(x',v') \in \Delta_N} |V^{N,\pi^N,\pi^N}(x',v') - V^{\pi^\infty,\pi^\infty}(x',v')| \to 0.$$
    
    For the second term, Lemma \ref{lem:distribution-convergence-pointwise} gives $\mu_1^{N,y} \xrightarrow{\mathbb{P}} \mu_1^{\bP,\pi^\infty}(v)$. By continuity of $V^{\pi^\infty,\pi^\infty}$ (Lemma \ref{lem:value-continuity}) and the continuous mapping theorem,
    $$V^{\pi^\infty,\pi^\infty}(y, \mu_1^{N,y}) \xrightarrow{\mathbb{P}} V^{\pi^\infty,\pi^\infty}(y, \mu_1^{\bP,\pi^\infty}(v)).$$
    Because $|V^{\pi^\infty,\pi^\infty}| \leq M$, convergence in probability together with uniform boundedness implies $L^1$ convergence, hence the second term also vanishes.
\end{proof}

\subsubsection{Convergence of Parametric Infima}

The proof of Theorem \ref{thm:convergence-mf} requires a result on the convergence of infima over varying compact sets with varying objective functions.

\begin{lemma}[Convergence of Parametric Infima]
	\label{lem:parametric-infima}
	Let $(X,d)$ be a metric space. For each $N \geq 1$, let $K_N \subseteq X$ be compact, and let $K \subseteq X$ be compact. Let $f_N: K_N \to \mathbb{R}$ and $f: K \to \mathbb{R}$ be continuous. Suppose
	\begin{enumerate}
		\item \emph{(Upper hemicontinuity)} For every sequence $(x_N) \subseteq X$ with $x_N \in K_N$, there exists a convergent subsequence $x_{N_k} \to x$ with $x \in K$;
		\item \emph{(Lower hemicontinuity)} For every $x \in K$, there exists a sequence $(x_N) \subseteq X$ with $x_N \in K_N$ such that $x_N \to x$;
		\item \emph{(Pointwise convergence)} For every sequence $x_N \to x$ with $x_N \in K_N$ and $x \in K$, we have $f_N(x_N) \to f(x)$.
	\end{enumerate}
	Then,
	$$\lim_{N \to \infty} \inf_{x \in K_N} f_N(x) = \inf_{x \in K} f(x).$$
\end{lemma}

\begin{proof}
	Let $v_N := \inf\limits_{x \in K_N} f_N(x)$ and $v := \inf\limits_{x \in K} f(x)$.
	
	\textbf{Lower bound: $\liminf\limits_{N \to \infty} v_N \geq v$.}
	Choosing a subsequence $(N_k)$  such that $v_{N_k} \to \liminf\limits_{N \to \infty} v_N$. For each $k$, because $K_{N_k}$ is compact and $f_{N_k}$ is continuous, there exists $x_{N_k} \in K_{N_k}$ such that $f_{N_k}(x_{N_k}) = v_{N_k}$. 
	
	Using Condition (i), the sequence $(x_{N_k})$ has a convergent subsequence $x_{N_{k_j}} \to x$ with $x \in K$.  Condition (iii) yields
	$$f(x) = \lim_{j \to \infty} f_{N_{k_j}}(x_{N_{k_j}}) = \lim_{j \to \infty} v_{N_{k_j}} = \liminf_{N \to \infty} v_N.$$
	
	Because $f(x) \geq v$, we have $\liminf\limits_{N \to \infty} v_N \geq v$.
	
	\textbf{Upper bound: $\limsup\limits_{N \to \infty} v_N \leq v$.}
	Because $K$ is compact and $f$ is continuous, the infimum $v$ is attained: there exists $x^* \in K$ such that $f(x^*) = v$. By condition (ii), there exists a sequence $x_N \in K_N$ such that $x_N \to x^*$. Using Condition (iii) yields
	$$\limsup_{N \to \infty} v_N \leq \limsup_{N \to \infty} f_N(x_N) = f(x^*) = v.$$
	
	Thus, combining both bounds, we have $\lim\limits_{N \to \infty} v_N = v$.
\end{proof}

\begin{remark}
	Conditions (i) and (ii) ensure that the sequence of sets $K_N$ converges to $K$ in an appropriate set-valued sense: every limit point of sequences from $K_N$ belongs to $K$, and every point in $K$ can be approximated by sequences from $K_N$. Lemma \ref{lem:parametric-infima} is related to classical results on epi-convergence in parametric optimization (cf.~\citet{rockafellar1998variational}, Chapter 7).
\end{remark}

\subsection{Proof of Theorem~\ref{thm:n-agent-existence}}

\begin{proof}
	We apply Kakutani's fixed-point theorem to the best-response correspondence.
	
	\textbf{Step 1: Strategy space.}
	Because $\Delta_N$ is finite, the $N$-agent strategy space is
	$$\Pi_N := \{\pi: \Delta_N \to \mathcal{P}(U)\} \cong \mathcal{P}(U)^{|\Delta_N|}.$$
	Under the product topology induced by the Wasserstein metric $d_{W_1}$ on each component, $\Pi_N$ is compact and convex.
	\vskip 5pt
    
	\textbf{Step 2: Best-response correspondence.}
	Define $B^N: \Pi_N \rightrightarrows \Pi_N$ by
	$$B^N(\pi) := \left\{\pi' \in \Pi_N : \pi'(x,v) \in \arg\max_{\sigma \in \mathcal{P}(U)} V^{N,\sigma \otimes_1 \pi,\pi}(x,v) \text{ for each } (x,v) \in \Delta_N\right\}.$$
	
	By Proposition \ref{prop:n-agent-one-shot}, a stationary equilibrium exists if and only if there exists $\pi^{N,*} \in B^N(\pi^{N,*})$.
	\vskip 5pt
    
	\textbf{Step 3: Joint continuity of the value function.}
    From the one-shot expansion \eqref{eq:one-shot-nagent-expanded},
    $$V^{N,\sigma \otimes_1 \pi,\pi}(x,v) = \inf_{\bP \in \Pf(v)} G_\bP(\sigma, \pi),$$
    where for each $\bP \in \Pf(v)$,
    $$G_\bP(\sigma, \pi) := \sum_{y \in \Dc} \int_U p(x,u)[y] \left[r(x,u,v,y) + \rho W^{N,\bP,\pi}(x,v,y)\right] \sigma(\dd u).$$
    We establish joint continuity of $V^{N,\sigma \otimes_1 \pi,\pi}(x,v)$ in $(\sigma,\pi)$ by showing that $G_\bP(\sigma,\pi)$ is jointly continuous uniformly over $\bP \in \Pf(v)$. Consider sequences $(\sigma^n, \pi^n) \to (\sigma, \pi)$ in $\mathcal{P}(U) \times \Pi_N$.
    
\textit{Continuity of $W^{N,\bP,\pi}$ in $\pi$.}
From \eqref{eq:auxiliary-nagent-W}, $W^{N,\bP,\pi}(x,v,y) = \mathbb{E}^{\bP}_{x,v}[V^{N,\pi,\pi}(y, \mu_1^{N,y})]$. Using the triangle inequality,
\begin{align*}
    |W^{N,\bP,\pi^n}(x,v,y) - W^{N,\bP,\pi}(x,v,y)| 
    &\leq \mathbb{E}[|V^{N,\pi^n,\pi^n}(y, \mu_1^{N,y}|_{\pi^n}) - V^{N,\pi,\pi}(y, \mu_1^{N,y}|_{\pi^n})|] \\
    &\quad + \mathbb{E}[|V^{N,\pi,\pi}(y, \mu_1^{N,y}|_{\pi^n}) - V^{N,\pi,\pi}(y, \mu_1^{N,y}|_{\pi})|],
\end{align*}
where $\mu_1^{N,y}|_{\sigma}$ denotes the conditional empirical distribution under strategy $\sigma$.

The first term is bounded by $\|V^{N,\pi^n,\pi^n} - V^{N,\pi,\pi}\|_{\infty,N} \to 0$ by Lemma \ref{lem:n-agent-value-continuity}. The second term has the same structure as $|H^{\pi^n}(y) - H^\pi(y)|$ in \hyperref[lem:operator-convergence-step1]{Step 1} of that proof with $F = V^{N,\pi,\pi}$; by the same coupling argument, it is bounded by $2M_V(N-1)L\epsilon_n \to 0$. Both bounds are uniform over $\bP \in \Pf(v)$ and $y \in \Dc$. 
\vskip 5pt	

    \textit{Continuity of $G_\bP(\sigma,\pi)$.}
    Define $g(u) := \sum_y p(x,u)[y][r(x,u,v,y) + \rho W^{N,\bP,\pi^n}(x,v,y)]$. Because $|W^{N,\bP,\pi^n}| \leq M_V := \frac{C_r}{1-\rho}$ (the uniform bound on value functions), by \eqref{eq:Pf} and Assumption \ref{ass:r}, $g$ is Lipschitz in $u$ with constant $L_g$ depending on $L$, $C_r$, $L_r$, and $M_V$. Then
    \begin{align*}
    |G_\bP(\sigma^n, \pi^n) - G_\bP(\sigma, \pi)| 
    &\leq \left|\int_U g(u) \, \sigma^n(\dd u) - \int_U g(u) \, \sigma(\dd u)\right| \\
    &\quad + \left|\int_U \sum_y p(x,u)[y] \rho [W^{N,\bP,\pi^n}(x,v,y) - W^{N,\bP,\pi}(x,v,y)] \, \sigma(\dd u)\right|.
    \end{align*}
    
    The first term is bounded by $L_g \cdot d_{W_1}(\sigma^n, \sigma)$ by the Kantorovich-Rubinstein duality. The second term is bounded by $\rho \max_y |W^{N,\bP,\pi^n}(x,v,y) - W^{N,\bP,\pi}(x,v,y)| \to 0$ uniformly over $\bP$ by the continuity of $W$ in $\pi$ established above.
    
    Thus $G_\bP(\sigma^n, \pi^n) \to G_\bP(\sigma, \pi)$ uniformly over $\bP \in \Pf(v)$.
    \vskip 5pt	
    
    \textit{Continuity of the value function.}
    Because $G_\bP(\sigma^n, \pi^n) \to G_\bP(\sigma, \pi)$ uniformly over the compact set $\Pf(v)$,
    $$V^{N,\sigma^n \otimes_1 \pi^n,\pi^n}(x,v) = \inf_{\bP \in \Pf(v)} G_\bP(\sigma^n, \pi^n) \to \inf_{\bP \in \Pf(v)} G_\bP(\sigma, \pi) = V^{N,\sigma \otimes_1 \pi,\pi}(x,v).$$
    \vskip 5pt	
    
    \textbf{Step 4: Properties of the best-response correspondence.}
    Applying Berge's maximum theorem to the continuous function $V^{N,\cdot \otimes_1 \pi,\pi}(x,v)$ over the compact set $\mathcal{P}(U)$, the best-response correspondence $B^N(\pi)$ has non-empty compact values and is upper hemicontinuous.
    
    To verify convexity, note that for each $\bP \in \Pf(v)$, the map $\sigma \mapsto G_\bP(\sigma, \pi)$ is linear in $\sigma$. Because $V^{N,\sigma \otimes_1 \pi,\pi}(x,v)$ is an infimum of linear functions, it is concave in $\sigma$. Therefore, the set of maximizers is convex, and $B^N(\pi)$ is convex-valued.
    \vskip 5pt	
    
	\textbf{Step 5: Fixed point.}
    By Kakutani's fixed-point theorem, there exists $\pi^{N,*} \in \Pi_N$ such that $\pi^{N,*} \in B^N(\pi^{N,*})$, i.e., for all $(x,v) \in \Delta_N$,
    $$\pi^{N,*}(x,v) \in \arg\max_{\sigma \in \mathcal{P}(U)} V^{N,\sigma \otimes_1 \pi^{N,*},\pi^{N,*}}(x,v).$$
    By Proposition \ref{prop:n-agent-one-shot}, $\pi^{N,*}$ is a stationary equilibrium.
\end{proof}

\subsection{Proof of Theorem~\ref{thm:mf-limit}}
\label{subsec:proof-approximation}

\begin{proof}
The proof proceeds in three steps: first, we show that 
the replicating strategy sequence satisfies the uniform convergence condition \eqref{eq:strategy-convergence}; second, we apply Lemma~\ref{lem:value-auxiliary-uniform} to establish that auxiliary 
functions are close for large $N$; finally, we combine these results with 
the equilibrium property of $\pi^*$ to verify the $\varepsilon$-Nash condition.

\textbf{Step 1: Uniform convergence of strategies.}
Because $\pi^{N,*} = \pi^*$ for all $N$, we have $d_{W_1}(\pi^{N,*}(x,v), \pi^*(x,v)) = 0$ for all $(x,v) \in \Delta_N$. Thus, the replicating strategy trivially satisfies the uniform convergence condition~\eqref{eq:strategy-convergence}.

\textbf{Step 2: Convergence of one-shot deviation values.}
By Lemma~\ref{lem:value-auxiliary-uniform}, for any $\delta > 0$, there exists $N_0$ such that for all $N \geq N_0$,
\begin{equation}
\label{eq:sup-W-bound}   
\sup_{(x,v) \in \Delta_N} \sup_{y \in \Dc} \sup_{\bP \in \Pf(v)} 
|W^{N,\bP,\pi^{N,*}}(x,v,y) - W^{\bP,\pi^{*}}(x,v,y)| \leq \delta. 
\end{equation}

For any $\pi^i \in \Pi$ and $(x,v) \in \Delta_N$, define for each $\bP \in \Pf(v)$:
\begin{align*}
F(\bP) &:= \sum_{y \in \Dc} \int_U p(x,u)[y] \left[r(x,u,v,y) + \rho W^{\bP,\pi^{*}}(x,v,y)\right] \pi^i(\dd u), \\
F^N(\bP) &:= \sum_{y \in \Dc} \int_U p(x,u)[y] \left[r(x,u,v,y) + \rho W^{N,\bP,\pi^{N,*}}(x,v,y)\right] \pi^i(\dd u).
\end{align*}
Then $V^{\pi^{i} \otimes_1 \pi^{*},\pi^{*}}(x, v) = \inf_{\bP \in \Pf(v)} F(\bP)$ and $V^{N,\pi^{i} \otimes_1 \pi^{N,*},\pi^{N,*}}(x, v) = \inf_{\bP \in \Pf(v)} F^N(\bP)$.

Using \eqref{eq:sup-W-bound} and $\sum_y \int_U p(x,u)[y] \pi^i(\dd u) = 1$, we have
$$|F(\bP) - F^N(\bP)| \leq \rho \delta \quad \forall \bP \in \Pf(v).$$
Because this bound is uniform over $\bP$, we obtain
$$|V^{\pi^{i} \otimes_1 \pi^{*},\pi^{*}}(x, v) - V^{N,\pi^{i} \otimes_1 \pi^{N,*},\pi^{N,*}}(x, v)| = |\inf_\bP F(\bP) - \inf_\bP F^N(\bP)| \leq \rho\delta.$$

\textbf{Step 3: Verification of $\varepsilon$-Nash equilibrium.}
Because $\pi^*$ is a mean-field equilibrium, by Proposition~\ref{prop:one_shot_eps},
$$V^{\pi^{i} \otimes_1 \pi^{*},\pi^{*}}(x, v) \leq V^{\pi^{*},\pi^{*}}(x, v) \quad \forall \pi^i \in \Pi, \, \forall (x,v) \in \Delta.$$

For any deviation $\pi^{N,i} \in \Pi_N$ in the $N$-agent game, we can naturally extend it to some $\pi^i \in \Pi$ (for instance, by defining arbitrary values on $\Delta \setminus \Delta_N$). Because the MFG value function $V^{\pi^i \otimes_1 \pi^*, \pi^*}(x,v)$
only depends on the values of $\pi^i$ at points in $\Delta_N$ when $(x,v) \in \Delta_N$, Step 2 applies with this extension. Combining yields
\begin{align*}
V^{N,\pi^{N,i} \otimes_1 \pi^{N,*},\pi^{N,*}}(x, v) 
&\leq V^{\pi^{N,i} \otimes_1 \pi^{*},\pi^{*}}(x, v) + \rho\delta \\
&\leq V^{\pi^{*},\pi^{*}}(x, v) + \rho\delta \\
&\leq V^{N,\pi^{N,*},\pi^{N,*}}(x, v) + 2\rho\delta.
\end{align*}

Choosing $\delta = \frac{1-\rho}{4\rho}\varepsilon$, we obtain that for all $N \geq N_0$,
$$V^{N,\pi^{N,i} \otimes_1 \pi^{N,*},\pi^{N,*}}(x, v) \leq V^{N,\pi^{N,*},\pi^{N,*}}(x, v) + \frac{1-\rho}{2}\varepsilon \quad \forall (x,v) \in \Delta_N.$$
By Proposition~\ref{prop:n-agent-one-shot}, $\pi^{N,*} = \pi^*$ is an $\varepsilon$-Nash equilibrium.
\end{proof}

\subsection{Proof of Theorem~\ref{thm:convergence-mf}}
\label{subsec:proof-convergence}

\begin{proof}
    We show that $\pi^\infty$ satisfies the one-shot optimality condition for the MFG, which by Proposition \ref{prop:one_shot_eps} establishes that $\pi^\infty$ is a mean-field equilibrium.
    
    \textbf{Step 1: Setting up the approximation.}
    Fix $(x,v) \in \Delta$ and $\pi^i \in \Pi$. For each $N \geq 2$, consider the approximating sequence $v(N,x) \in \Delta_N$ satisfying \eqref{eq:v-N-x-convergence}. Because $\pi^N$ is an $N$-agent equilibrium, by Proposition \ref{prop:n-agent-one-shot}, the one-shot optimality condition holds:
    \begin{equation}
        \label{eq:n-agent-oneshot-convergence}
        V^{N,\pi^i \otimes_1 \pi^N,\pi^N}(x,v(N,x)) \leq V^{N,\pi^N,\pi^N}(x,v(N,x))
    \end{equation}
    Here we use the fact that $\pi^i|_{\Delta_N} \in \Pi_N$, and the $N$-agent one-shot optimality holds for all deviations in $\Pi_N$.
    
    \textbf{Step 2: Convergence of one-shot values.}
    From the one-shot expansions \eqref{eq:one-shot-mf-expanded} and \eqref{eq:one-shot-nagent-expanded}, define
    \begin{align*}
        F(\bP) &:= \sum_{y \in \Dc} \int_U p(x,u)[y] \left[r(x,u,v,y) + \rho W^{\bP,\pi^\infty}(x,v,y)\right] \pi^i(\dd u), \\
        F^N(\bP_N) &:= \sum_{y \in \Dc} \int_U p(x,u)[y] \left[r(x,u,v(N,x),y) + \rho W^{N,\bP_N,\pi^N}(x,v(N,x),y)\right] \pi^i(\dd u).
    \end{align*}
    Then $V^{\pi^i \otimes_1 \pi^\infty,\pi^\infty}(x,v) = \inf_{\bP \in \Pf(v)} F(\bP)$ and $V^{N,\pi^i \otimes_1 \pi^N,\pi^N}(x,v(N,x)) = \inf_{\bP_N \in \Pf(v(N,x))} F^N(\bP_N)$.

    We apply Lemma \ref{lem:parametric-infima} to establish convergence of the infima. By Assumption \ref{ass:uncertainty}, the correspondence $\Pf(\cdot)$ is continuous (both upper and lower hemicontinuous), so conditions (i) and (ii) follow from $v(N,x) \to v$. For condition (iii), consider any convergent sequence $\bP_N \to \bP$ with $\bP_N \in \Pf(v(N,x))$ and $\bP \in \Pf(v)$. The reward term converges by Assumption \ref{ass:r} and $v(N,x) \to v$. The continuation term converges by Lemma \ref{lem:auxiliary-convergence-pointwise}, which gives $W^{N,\bP_N,\pi^N}(x,v(N,x),y) \to W^{\bP,\pi^\infty}(x,v,y)$ for each $y \in \Dc$. Therefore $F^N(\bP_N) \to F(\bP)$, and Lemma \ref{lem:parametric-infima} yields
    $$\lim_{N \to \infty} V^{N,\pi^i \otimes_1 \pi^N,\pi^N}(x,v(N,x)) = V^{\pi^i \otimes_1 \pi^\infty,\pi^\infty}(x,v).$$
    The same argument with $\pi^i = \pi^N$ (noting that $\pi^N \to \pi^\infty$) gives
    $$\lim_{N \to \infty} V^{N,\pi^N,\pi^N}(x,v(N,x)) = V^{\pi^\infty,\pi^\infty}(x,v).$$
    
    \textbf{Step 3: Preserving the inequality.}
    Taking limits in \eqref{eq:n-agent-oneshot-convergence} yields
    $$V^{\pi^i \otimes_1 \pi^\infty,\pi^\infty}(x,v) \leq V^{\pi^\infty,\pi^\infty}(x,v).$$
    Because this holds for all $(x,v) \in \Delta$ and all $\pi^i \in \Pi$, by Proposition \ref{prop:one_shot_eps}, $\pi^\infty$ is a mean-field equilibrium.
\end{proof}

\section{Conclusion}
\label{sec:conclusion}

This paper studies discrete-time, finite-state MFGs under model uncertainty. We propose a framework to address the fundamental challenge that arises when agents lack precise knowledge of the state transition probabilities---namely, that the population distribution evolves stochastically rather than deterministically. We extend the strategy space to allow controls to depend on both the individual state and population distribution, so that agents can react adaptively to the stochastic evolution of the population. We provide a rigorous dynamic programming characterization for the agents' robust utility functions, where adversarial nature selects the worst-case transition probabilities from an uncertainty set that may depend on the state-distribution pair. This formulation captures the economic intuition that different environments may yield different worst-case scenarios. 

Our main theoretical results establish the precise asymptotic relationship between $N$-agent games and their mean-field limits under model uncertainty. Theorem~\ref{thm:mf-limit} shows that mean-field equilibria are $\varepsilon$-Nash equilibria for $N$-agent games (for large $N$), justifying the MFG framework as a practical approximation tool for large population games. Theorem~\ref{thm:convergence-mf} proves the converse result: equilibria of $N$-agent games converge to those of MFGs as $N$ increases. Finally, Theorem~\ref{thm:n-agent-existence} guarantees the existence of stationary equilibria for $N$-agent games, which, combined with the convergence results, provides an indirect pathway to establishing the existence of mean-field equilibria.

One important open question concerns the direct existence of mean-field equilibria. As discussed in Subsection~\ref{sub:main}, the strategy space $\Pi_c$ lacks compactness, and the best-response correspondence may not admit a continuous selection, preventing standard fixed-point arguments. Whether existence can be established under additional structural assumptions, or whether counterexamples exist showing nonexistence, remain interesting directions for future research.

\paragraph{Acknowledgements.} Zongxia Liang is supported by the National Natural Science Foundation of China (no.12271290). The authors thank the members of the group of Mathematical Finance and Actuarial Science at the Department of Mathematical Sciences, Tsinghua University, for their helpful feedback and conversations.

\bibliographystyle{abbrvnat}
\bibliography{ref}

\section*{Appendix}
\appendix

\section{Alternative Assumptions} 
\label{app:ass}

When we introduce the MFG under model uncertainty in Subsection \ref{sub:MFG}, we impose Assumption \ref{ass:uncertainty} on the functional $\Pf$, which maps any population distribution $v \in \Pc(\Dc)$ into a set of transition probabilities varying over all state-control pairs $(x, u) \in \Dc \times U$. In its definition \eqref{eq:Pf}, we require that for all given triples $(x, u, v)$, all possible transition probabilities must be in the set $\Ff(x, u, v)$, which is a subset of $\Pc(\Dc)$ and satisfies the $L$-Lipschitz condition in \eqref{eq:F_Lip}. In this Appendix, we show that Assumption \ref{ass:uncertainty} on $\Pf$ can be equivalently formulated in terms of $\Ff$.

\begin{proposition}
    \label{prop:mapping}
    Assume that the set-valued mapping $\Ff$ admits a ``ball structure'' in the sense that there exist two functions $\mathtt{c}: \Dc \times U \times \Pc(\Dc) \to \Pc(\Dc)$ (``center'') and $\mathtt{R}: \Dc \times U \times \Pc(\Dc) \to \Rb_+$ (``radius'') such that, for all $(x, u, v) \in \Dc \times U \times \Pc(\Dc)$,  
    \begin{align*}
        \Ff(x, u, v) = \left\{ p \in \Pc(\Dc) : d_{W_1}(p, \, \mathtt{c}(x, u, v)) \le \mathtt{R}(x, u, v) \right\}.
    \end{align*} 
    Additionally, suppose that:
    \begin{enumerate}
        \item[(i)] $\mathtt{c}$ is $L$-Lipschitz continuous:
        \begin{align*}
            d_{W_1}(\mathtt{c}(x_1, u_1, v_1), \mathtt{c}(x_2, u_2, v_2)) \leq L\big(d_\Dc(x_1, x_2) + |u_1- u_2| + d_{W_1}(v_1, v_2)\big);
        \end{align*}
        \item[(ii)] $\mathtt{R}$ is continuous and $\mathtt{R}(x, u, v) \geq \mathtt{R}_{\min} > 0$ for all $(x, u, v) \in \Dc \times U \times \Pc(\Dc)$.
    \end{enumerate}
    Then, Assumption \ref{ass:uncertainty} holds. 
\end{proposition}

In the subsequent proof, we extensively use properties of the Wasserstein-1 distance $d_{W_1}$ and introduce some preliminary results. For all $p, q \in \Pc(\Dc)$, recall that the Wasserstein-1 distance 
admits the Kantorovich-Rubinstein dual representation:
\begin{equation}\label{eq:KR_dual}
	d_{W_1}(p, q) = \sup_{f \in \text{Lip}_1} \left|\int f \, \dd p - \int f \, \dd q \right|,
\end{equation}
where $\text{Lip}_1$ denotes the set of 1-Lipschitz functions $f: \Dc \to \mathbb{R}$. 
This representation is crucial for establishing Lipschitz continuity of convex combinations.
In particular, for all $\lambda \in [0,1]$ and probability measures $p_1, p_2, q_1, q_2 \in \Pc(\Dc)$, 
\begin{equation}\label{eq:convex_combo_lip}
	d_{W_1}(\lambda p_1 + (1- \lambda) p_2, \lambda q_1 + (1-\lambda) q_2) 
	\leq \lambda d_{W_1}(p_1, q_1) + (1-\lambda) d_{W_1}(p_2, q_2),
\end{equation}
which follows from the linearity of the integral in the dual representation.

\begin{proof}[Proof of Proposition \ref{prop:mapping}]
	We establish the four properties in Assumption~\ref{ass:uncertainty} sequentially.
	
	\medskip
	\textbf{Step 1: Non-emptiness.} 
	The center function $\mathtt{c}(\cdot, \cdot, \nu)$ itself provides a natural 
	element of $\Pf(\nu)$. Define $\bP^0: \Dc \times U \to \Pc(\Dc)$ by 
	$\bP^0(x,u) = \mathtt{c}(x, u, \nu)$ for all $(x,u)$.
	
	By Condition (i) of the proposition, $\bP^0$ is $L$-Lipschitz continuous
	\[
	d_{W_1}(\bP^0(x_1,u_1), \bP^0(x_2,u_2)) = d_{W_1}(\mathtt{c}(x_1, u_1, \nu), 
	\mathtt{c}(x_2, u_2, \nu)) \leq L(d_\Dc(x_1,x_2) + |u_1-u_2|).
	\]
	Moreover, by the ball structure, $\mathtt{c}(x, u, \nu)$ is the center of 
	$\Ff(x, u, \nu)$, hence $\bP^0(x,u) = \mathtt{c}(x, u, \nu) \in \Ff(x, u, \nu)$ 
	for all $(x,u)$. Therefore, $\bP^0 \in \Pf(\nu)$, establishing non-emptiness.
	
	\medskip
	\textbf{Step 2: Convexity.} 
	Let $\bP_1, \bP_2 \in \Pf(\nu)$ and $\lambda \in [0,1]$. Define 
	$\bP_\lambda = \lambda \bP_1 + (1-\lambda) \bP_2$, i.e., 
	$\bP_\lambda(x,u) = \lambda \bP_1(x,u) + (1-\lambda) \bP_2(x,u)$ for all $(x,u)$.
	
	\textit{Membership in $\Ff(x, u, \nu)$:} Because $\Ff(x, u, \nu)$ has the ball 
	structure and is therefore convex, we have
	\[
	\bP_\lambda(x,u) = \lambda \bP_1(x,u) + (1-\lambda) \bP_2(x,u) \in \Ff(x, u, \nu)
	\]
	for each $(x,u) \in \Dc \times U$.
	
	\textit{Lipschitz continuity:} For any $(x_1,u_1), (x_2,u_2) \in \Dc \times U$, 
	using \eqref{eq:convex_combo_lip},
	\begin{align*}
		&d_{W_1}(\bP_\lambda(x_1,u_1), \bP_\lambda(x_2,u_2)) \\
		&= d_{W_1}(\lambda \bP_1(x_1,u_1) + (1-\lambda) \bP_2(x_1,u_1), 
		\lambda \bP_1(x_2,u_2) + (1-\lambda) \bP_2(x_2,u_2)) \\
		&\leq \lambda d_{W_1}(\bP_1(x_1,u_1), \bP_1(x_2,u_2)) 
		+ (1-\lambda) d_{W_1}(\bP_2(x_1,u_1), \bP_2(x_2,u_2)) \\
		&\leq \lambda L(d_\Dc (x_1,x_2) + |u_1-u_2|)  
		+ (1-\lambda) L(d_\Dc (x_1,x_2) + |u_1-u_2|) \\
		&= L(d_\Dc (x_1,x_2) + |u_1-u_2|).
	\end{align*}
	
	Therefore, $\bP_\lambda \in \Pf(\nu)$, establishing convexity.
	
	\medskip
	\textbf{Step 3: Compactness.} 
	Consider any sequence $\{\bP_n\} \subset \Pf(\nu)$. Each $\bP_n$ is an 
	$L$-Lipschitz continuous function from the compact space $\Dc \times U$ to 
	$\Pc(\Dc)$. Therefore, the family $\{\bP_n\}$ is equicontinuous. Moreover, 
	because $\Pc(\Dc)$ is a finite-dimensional compact space, the family is 
	uniformly bounded.
	
	By the Arzel\`a-Ascoli theorem, $\{\bP_n\}$ is sequentially compact in the 
	topology of uniform convergence. Therefore, there exists a subsequence 
	$\{\bP_{n_k}\}$ and a continuous function $\bP: \Dc \times U \to \Pc(\Dc)$ 
	such that $\bP_{n_k} \to \bP$ uniformly, i.e., $d_\infty(\bP_{n_k}, \bP) \to 0$.
	
	We now verify that $\bP \in \Pf(\nu)$.
	
	\textit{Membership in $\Ff(x, u, \nu)$:} 
	For each $(x,u)$, because $\bP_{n_k}(x,u) \in \Ff(x, u, \nu)$ and 
	$\Ff(x, u, \nu)$ is closed, the limit 
	$\bP(x,u) = \lim\limits_{k\to\infty} \bP_{n_k}(x,u)$ belongs to $\Ff(x, u, \nu)$.
	
	\textit{Lipschitz continuity:} 
	For any $(x_1,u_1), (x_2,u_2) \in \Dc \times U$, using the triangle inequality, we have
	\begin{align*}
		d_{W_1}(\bP(x_1,u_1), \bP(x_2,u_2)) 
		&\leq d_{W_1}(\bP(x_1,u_1), \bP_{n_k}(x_1,u_1)) \\
		&\quad + d_{W_1}(\bP_{n_k}(x_1,u_1), \bP_{n_k}(x_2,u_2)) \\
		&\quad + d_{W_1}(\bP_{n_k}(x_2,u_2), \bP(x_2,u_2)) \\
		&\leq 2d_\infty(\bP_{n_k}, \bP) + L(d_\Dc (x_1,x_2) + |u_1-u_2|).
	\end{align*}
	Taking $k \to \infty$ yields 
	$d_{W_1}(\bP(x_1,u_1), \bP(x_2,u_2)) \leq L(d_\Dc (x_1,x_2) + |u_1-u_2|)$.
	
	Therefore, $\bP \in \Pf(\nu)$. This shows that $\Pf(\nu)$ is sequentially 
	compact. Because $(\Pc(\Dc)^{\Dc \times U}, d_\infty)$ is a metric space, 
	sequential compactness is equivalent to compactness. Thus, $\Pf(\nu)$ is compact.
	
	\medskip
    \textbf{Step 4: Upper semicontinuity.} 
	Let $\nu_n \to \nu$ in $\Pc(\Dc)$. Suppose $\bP_n \in \Pf(\nu_n)$ with 
	$\bP_n \to \bP$ in $d_\infty$. We need to show $\bP \in \Pf(\nu)$.
	
	For each $(x,u)$, we have $\bP_n(x,u) \to \bP(x,u)$ and 
	$\bP_n(x,u) \in \Ff(x, u, \nu_n)$. This means that
	$$
	d_{W_1}(\bP_n(x,u), \mathtt{c}(x, u, \nu_n)) \leq \mathtt{R}(x, u, \nu_n).
	$$
	By the continuity of $\mathtt{c}$ and $\mathtt{R}$ (from the conditions (i) and (ii)), 
	we have $\mathtt{c}(x, u, \nu_n) \to \mathtt{c}(x, u, \nu)$ and 
	$\mathtt{R}(x, u, \nu_n) \to \mathtt{R}(x, u, \nu)$. Taking limits in the 
	inequality above,
	\[
	d_{W_1}(\bP(x,u), \mathtt{c}(x, u, \nu)) \leq \mathtt{R}(x, u, \nu),
	\]
	which implies $\bP(x,u) \in \Ff(x, u, \nu)$.
	
	The $L$-Lipschitz continuity of $\bP$ follows from the argument in 
	Step 3. Therefore, $\bP \in \Pf(\nu)$, 
	establishing upper semicontinuity.

    \medskip
    \textbf{Step 5: Lower semicontinuity.} 
    Given $\bP \in \Pf(\nu)$ and $\nu_n \to \nu$, we construct a sequence 
    $\bP_n \in \Pf(\nu_n)$ such that $\bP_n \to \bP$ in $d_\infty$.

    Define $\bP_n: \Dc \times U \to \Pc(\Dc)$ by
    \begin{equation}\label{eq:Pn_construction}
    \bP_n(x,u) := (1-\lambda_n) \mathtt{c}(x, u, \nu_n) + \lambda_n \bP(x,u),
    \end{equation}
    where the weight $\lambda_n$ is defined as
    \begin{equation}\label{eq:lambda_uniform}
    \lambda_n = \min_{(x,u) \in \Dc \times U} \min\left\{1, 
    \frac{\mathtt{R}(x, u, \nu_n)}{\mathtt{R}(x, u, \nu) + L \cdot d_{W_1}(\nu, \nu_n)}\right\}.
    \end{equation}

    \textit{Verification of membership:} 
    Using the triangle inequality,
    \begin{align*}
    d_{W_1}(\bP_n(x,u), \mathtt{c}(x, u, \nu_n)) 
    &= \lambda_n \cdot d_{W_1}(\bP(x,u), \mathtt{c}(x, u, \nu_n)) \\
    &\leq \lambda_n \cdot [d_{W_1}(\bP(x,u), \mathtt{c}(x, u, \nu)) 
    + d_{W_1}(\mathtt{c}(x, u, \nu), \mathtt{c}(x, u, \nu_n))] \\
    &\leq \lambda_n \cdot [\mathtt{R}(x, u, \nu) + L \cdot d_{W_1}(\nu, \nu_n)],
    \end{align*}
    where we used $\bP(x,u) \in \Ff(x, u, \nu)$ and the Lipschitz continuity of $\mathtt{c}$ from Condition (i). By the choice of $\lambda_n$ in~\eqref{eq:lambda_uniform},
    \[
    d_{W_1}(\bP_n(x,u), \mathtt{c}(x, u, \nu_n)) 
    \leq \lambda_n [\mathtt{R}(x, u, \nu) + L \cdot d_{W_1}(\nu, \nu_n)] 
    \leq \mathtt{R}(x, u, \nu_n),
    \]
    ensuring $\bP_n(x,u) \in \Ff(x, u, \nu_n)$ for all $(x,u)$.

    \textit{Verification of Lipschitz continuity:} 
	Because $\lambda_n$ is a constant across $(x,u)$,
    using~\eqref{eq:convex_combo_lip},
	\begin{align*}
		d_{W_1}(\bP_n(x_1,u_1), \bP_n(x_2,u_2)) 
		&= d_{W_1}((1-\lambda_n)\mathtt{c}(x_1, u_1, \nu_n) + \lambda_n \bP(x_1,u_1), \\
		&\qquad\qquad (1-\lambda_n)\mathtt{c}(x_2, u_2, \nu_n) + \lambda_n \bP(x_2,u_2)) \\
		&\leq (1-\lambda_n) d_{W_1}(\mathtt{c}(x_1, u_1, \nu_n), 
		\mathtt{c}(x_2, u_2, \nu_n)) \\
		&\quad + \lambda_n d_{W_1}(\bP(x_1,u_1), \bP(x_2,u_2)) \\
		&\leq L(d_\Dc(x_1,x_2) + |u_1-u_2|).
	\end{align*}
	
	Therefore, $\bP_n \in \Pf(\nu_n)$.

    \textit{Convergence in $d_\infty$:} 
    We show that $\lambda_n \to 1$. For each $(x,u) \in \Dc \times U$, the ratio
    $$r_n(x,u) = \frac{\mathtt{R}(x, u, \nu_n)}{\mathtt{R}(x, u, \nu) + L \cdot d_{W_1}(\nu, \nu_n)}.$$

    We establish that $r_n(x,u) \to 1$ uniformly over $(x,u) \in \Dc \times U$. Because $\mathtt{R}$ is continuous on the compact set $\Dc \times U \times \Pc(\Dc)$, it is uniformly continuous. Hence $\mathtt{R}(x,u,\nu_n) \to \mathtt{R}(x,u,\nu)$ uniformly over $(x,u)$ as $\nu_n \to \nu$. Since the denominator satisfies $\mathtt{R}(x,u,\nu) + L \cdot d_{W_1}(\nu,\nu_n) \geq \mathtt{R}_{\min} > 0$ uniformly, and converges uniformly to $\mathtt{R}(x,u,\nu)$, we conclude that $r_n(x,u) \to 1$ uniformly over $(x,u) \in \Dc \times U$.

    Therefore, $\lambda_n = \min_{(x,u)} \min\{1, r_n(x,u)\} \to 1$.

    Now,
    \begin{align*}
    d_\infty(\bP_n, \bP) 
    &= \sup_{(x,u)} (1-\lambda_n) \cdot d_{W_1}(\mathtt{c}(x, u, \nu_n), \bP(x,u)) \\
    &\leq (1-\lambda_n) \sup_{(x,u)} [\mathtt{R}(x, u, \nu) + L \cdot d_{W_1}(\nu_n, \nu)] \to 0,
    \end{align*}
    where we used $(1-\lambda_n) \to 0$ and the boundedness of $\mathtt{R}$ on the compact set $\Dc \times U \times \{\nu\}$. This establishes lower semicontinuity.
\end{proof}

\end{document}